\documentclass[a4paper, 10pt]{article}
\pdfoutput = 1
\usepackage[UKenglish]{babel}
\usepackage{csquotes}
\usepackage{fullpage}
\usepackage{amsmath, amssymb, amsthm, amsfonts, mathtools, scalefnt}
\usepackage{mathrsfs} 
\usepackage[dvipsnames]{xcolor}
\usepackage{tikz}
\usepackage{tikz-cd}
\usepackage{enumitem}
\usepackage{colonequals}
\usepackage{thmtools}
\usepackage{booktabs}
\usepackage{parskip}
\usepackage{hyperref}
\hypersetup{colorlinks = true, linkcolor = black, citecolor = blue, urlcolor = black}

\usepackage{footmisc}
\usepackage{todonotes}
\usepackage{url}
\usepackage{microtype}
\usepackage{lmodern}
\usepackage{cleveref}
\usepackage{parskip}

\usepackage[style = alphabetic, backend = biber, maxbibnames = 99, datamodel = mrnumber]{biblatex}
\DeclareFieldFormat{mrnumber}{%
  MR\addcolon\space
  \ifhyperref
    {\href{http://www.ams.org/mathscinet-getitem?mr=#1}{\nolinkurl{#1}}}
    {\nolinkurl{#1}}}

\renewbibmacro*{doi+eprint+url}{%
    \iftoggle{bbx:doi}
      {\printfield{doi}}
      {}%
    \newunit\newblock
    \printfield{mrnumber}%
    \newunit\newblock
    \iftoggle{bbx:eprint}
      {\usebibmacro{eprint}}
      {}%
    \newunit\newblock
    \iftoggle{bbx:url}
      {\usebibmacro{url+urldate}}
      {}}

\newcommand{\renewtheorem}[1]{%
  \expandafter\let\csname #1\endcsname\relax
  \expandafter\let\csname c@#1\endcsname\relax
  \expandafter\let\csname end#1\endcsname\relax
  \newtheorem{#1}%
}

\theoremstyle{plain}

\newtheorem{alphatheorem}{Theorem}

\declaretheoremstyle[spaceabove = 3pt, spacebelow = 3pt, bodyfont = \itshape]{mystyle}
\declaretheoremstyle[spaceabove = 3pt, spacebelow = 3pt, bodyfont = \normalfont]{remark}

\declaretheorem[numberwithin = section, style = mystyle]{theorem}
\declaretheorem[sibling = theorem, style = mystyle]{proposition}
\declaretheorem[sibling = theorem, style = mystyle]{lemma}
\declaretheorem[sibling = theorem, style = mystyle]{corollary}

\declaretheorem[sibling = theorem, style = remark]{definition}

\declaretheorem[sibling = theorem, style = remark]{remark}

\newcommand{\tH}{\ensuremath{\mathrm{H}}}

\newcommand{\sfA}{\mathsf{A}}\newcommand{\sfB}{\mathsf{B}}\newcommand{\sfb}{\mathsf{b}}\newcommand{\sfC}{\mathsf{C}}\newcommand{\sfD}{\mathsf{D}}\newcommand{\sfO}{\mathsf{O}}\newcommand{\sfS}{\mathsf{S}}\newcommand{\sfT}{\mathsf{T}}

\newcommand{\bbN}{\mathbb{N}}\newcommand{\bbP}{\mathbb{P}}\newcommand{\bbR}{\mathbb{R}}\newcommand{\bbS}{\mathbb{S}}\newcommand{\bbZ}{\mathbb{Z}}

\newcommand{\mff}{\mathfrak{f}}\newcommand{\mfg}{\mathfrak{g}}\newcommand{\mfh}{\mathfrak{h}}\newcommand{\mfi}{\mathfrak{i}}\newcommand{\mfj}{\mathfrak{j}}\newcommand{\mfm}{\mathfrak{m}}\newcommand{\mfp}{\mathfrak{p}}

\newcommand{\calE}{\mathcal{E}}\newcommand{\calF}{\mathcal{F}}\newcommand{\calH}{\mathcal{H}}\newcommand{\calL}{\mathcal{L}}\newcommand{\calO}{\mathcal{O}}\newcommand{\calX}{\mathcal{X}}

\newcommand{\bfr}{\mathbf{r}}

\newcommand{\bfL}{\mathbf{L}}

\newcommand{\bfR}{\mathbf{R}}

\newcommand{\alg}{\ensuremath{\mathrm{alg}}}

\newcommand{\coh}{\mathsf{coh}}
\newcommand{\compact}{\ensuremath{\mathrm{c}}}
\newcommand{\Coker}{\ensuremath{\mathrm{Coker}}}
\newcommand{\bd}{\mathsf{D}^{\mathsf{b}}}

\DeclareMathOperator\gldim{\ensuremath{\mathrm{gldim}}}
\newcommand{\End}{\ensuremath{\mathrm{End}}}
\newcommand{\Ext}{\ensuremath{\mathrm{Ext}}}
\newcommand\field{\mathbf{k}}

\DeclareMathOperator{\hocolim}{\mathrm{hocolim}}

\newcommand{\Hom}{\ensuremath{\mathrm{Hom}}}
\renewcommand{\Im}{\ensuremath{\mathrm{Im}}}

\newcommand{\intend}{\mathcal{E}\kern -.5pt nd}
\newcommand{\intext}{{\mathcal{E}\kern -.5pt xt}}
\newcommand{\inthom}{\mathcal{H}\kern -.5pt om}

\newcommand{\Mat}{\ensuremath{\mathrm{Mat}}}
\DeclareMathOperator{\modules}{\mathsf{mod}}

\newcommand{\ncSch}{\mathsf{\mathbf{ncSch}}}
\newcommand{\order}[1]{\mathscr{#1}}

\newcommand{\perf}{\mathsf{perf}}
\newcommand{\pr}{\ensuremath{\mathrm{pr}}}

\newcommand{\QCoh}{\mathsf{QCoh}}

\DeclareMathOperator{\rad}{\ensuremath{\mathrm{rad}}}
\DeclareMathOperator{\rk}{\ensuremath{\mathrm{rk}}}

\newcommand{\relspec}[2]{\underline{\smash{\mathrm{Spec}}}_{#1}\!\left(#2\right)}
\DeclareMathOperator{\Spec}{Spec}

\newcommand\stack[1]{\mathcal{#1}}

\newcommand{\sbm}[1]{{\let\amp=&\left(\begin{smallmatrix}#1\end{smallmatrix}\right)}}

\hypersetup{pageanchor=true}

\title{\Large\textbf{Categorical absorption for hereditary orders}}
\author{Thilo Baumann}
\date{\today}
\begin{document}
\maketitle

\begin{abstract}
We show that Kuznetsov--Shinder's notion of deformation absorption of singularities
leads to a new approach for studying the bounded derived category of a hereditary order on a curve.
The starting point is a hereditary order which can be interpreted as a smoothing 
of the finite-dimensional algebra obtained from the restriction to a ramified point.
We construct a triangulated subcategory inside the derived category of this finite-dimensional algebra
which provides a deformation absorption of singularities.
This allows us to obtain a semiorthogonal decomposition of the bounded derived category of the hereditary order,
which is in addition linear over the base.
\end{abstract}

\tableofcontents
\section{Introduction}
Hereditary orders over a curve form a well-studied class of sheaves of (noncommutative) algebras.
Their classification up to étale-local isomorphism \cite[Theorem 39.14]{MR393100},
and up to Morita equivalence \cite[Proposition 7.7]{MR3695056} is well-established.

From the geometric point of view, \cite[Corollary 7.8]{MR2018958} provides a dictionary between 
\begin{enumerate}[label = \alph*)]
    \item hereditary orders~$\order{A}$ on a smooth curve~$C$, and
    \item smooth root stacks~$\stack{C}$ with coarse moduli space~$C$,
\end{enumerate}
in the sense
that there is
an equivalence~$\coh(C,\order{A}) \simeq \coh(\stack{C})$
of abelian categories.
The dictionary identifies ramification points of
the order~$\order{A}$
with the points with non-trivial stabilizer of
the root stack~$\stack{C}$.

The bounded derived category~$\bd(C,\order{A}) \simeq \bd(\stack{C})$ has been examined
on both sides of the dictionary
with various objectives in mind.
In particular, when trying to decompose these categories,
one has the following results if~$C$ is projective.
\begin{itemize}
    \item If~$C = \bbP^1$, the stacky curve~$\stack{C}$ is a weighted projective line,
    and~$\bd(\stack{C})$ admits a tilting bundle \cite[Proposition 4.1]{MR915180}.
    This was recently generalized to a certain class of hereditary orders
    over non-algebraically closed fields in \cite[Theorem 3.12]{2411.06222}.
    \item Among other things, \cite[Theorem 1.2]{MR3436544}
    construct a semiorthogonal decomposition of~$\bd(\stack{C})$.
\end{itemize}

We propose a novel approach,
using the framework of Kuznetsov--Shinder \cite{MR4698898}
to decompose~$\bd(C,\order{A})$,
by viewing~$(C,\order{A})$ as a family of finite-dimensional algebras over the curve~$C$.
The idea fits into the general perspective of studying the bounded derived category
of families of varieties over a base scheme, see \cite{MR4680279} for a survey.
It provides the first example of a deformation absorption of singularities \cite{MR4698898} in a noncommutative setting.

To state the first main result of the paper,
let~$C$ be a smooth curve over an algebraically closed field~$\field$ of characteristic zero.
Consider a hereditary~$\calO_C$-order~$\order{A}$
with ramification locus~$\Delta_\order{A} = \{o\}$
and ramification index~$r\in \bbZ_{\ge 1}$.
Hence, the algebra~$\order{A}$ is Azumaya on~$C\setminus\{o\}$.
The algebra~$\order{A}(o) \colonequals \order{A}\otimes_C\Spec\field(o)$ is described in \Cref{lem:morita-equivalence-to-cyclic-algebra}.
The first main result is the existence of a triangulated subcategory
in~$\bd(\order{A}(o))$ which absorbs singularities.
\begin{alphatheorem}[\Cref{thm:absorption-of-singularities}]\label{thm:absorption-of-singularities-intro}
    The sequence~$(S_1,\ldots, S_{r-1})$ of simple~$\order{A}(o)$-modules is semiorthogonal in~$\bd(\order{A}(o))$ and 
    absorbs singularities, i.e.~the triangulated subcategory
    \begin{equation}\label{eq:sequence-of-p-infty-objects-introdcution}
        \sfS = \langle S_1,\ldots, S_{r-1}\rangle \subset \bd(\order{A}(o))
    \end{equation}
    is admissible and both of its complements~${}^\perp \sfS$ and~$\sfS^\perp$ are smooth and proper.
\end{alphatheorem}
More precisely, we show in \Cref{lem:p-infty-2-objects} that
the sequence \eqref{eq:sequence-of-p-infty-objects-introdcution}
forms a semiorthogonal sequence of~$\bbP^{\infty,2}$-objects.
The proof uses the representation theory of finite-dimensional algebras.
Appealing to a noncommutative version of \cite[Theorem 1.8]{MR4698898}
we use~$\sfS$ to provide a semiorthogonal decomposition of~$\bd(C,\order{A})$.
\begin{alphatheorem}[\Cref{thm:deformation-absorption}]\label{thm:deformation-absorption-intro}
    Let~$\mfi_o\colon (\Spec \field(o),\order{A}(o))\to (C,\order{A})$ be the inclusion of~$o\in C$.
    There is a strong~$C$-linear semiorthogonal decomposition
    \begin{equation}\label{eq:main-sod-theoremA}
        \bd(C,\mathscr{A}) =
        \langle \mfi_{o,*}S_{1},\ldots,  \mfi_{o,*}S_{r-1}, \sfD\rangle,
    \end{equation}
    such that
    \begin{enumerate}[label = \roman*)]
        \item the sequence~$(\mfi_{o,*}S_{1},\ldots,  \mfi_{o,*}S_{r-1})$ is exceptional,
        \item the admissible subcategory~$\sfD$ is smooth and proper over~$\bd(C)$.
    \end{enumerate}
\end{alphatheorem}
In fact, we will show in \Cref{lem:embedding-the-curve} that~$\sfD$ is equivalent to~$\bd(C)$.

The existence of a semiorthogonal decomposition \eqref{eq:main-sod-theoremA}
can also be deduced in an indirect way
from the stacks--orders dictionary \cite[Corollary 7.8]{MR2018958}
and the semiorthogonal decomposition of \cite[Theorem 4.7]{MR3573964} for root stacks.
We explain in \Cref{sec:dictionary-orders-stacks} how our results translate to the stacky side in \cite[Theorem 4.7]{MR3573964}.

\paragraph{Categorical absorption and deformation absorption of singularities.}
We outline the idea of categorical absorption and
deformation absorption of singularities from \cite{MR4698898}.

Given a flat family~$\calX\to (C,o)$ over a smooth pointed curve
with a single singular fiber~$\calX_o = \calX\times_C\Spec\field(o)$
and smooth total space~$\calX$,
there is sometimes an interplay
of a certain semiorthogonal decomposition~$\bd(\calX_o) = \langle \sfS,{}^{\perp}\sfS\rangle$ of the fiber
and a~$C$-linear semiorthogonal decomposition~$\bd(\calX) = \langle \sfA_{\sfS},\sfB\rangle$ of the total space
in the following sense.
\begin{enumerate}[label = \alph*)]
    \item As in \Cref{thm:absorption-of-singularities-intro},
    the category~$\sfS\subset \bd(\calX_o)$ is admissible and
    both of its complements~${}^\perp \sfS$ and~$\sfS^\perp$ are smooth and proper.
    In this case \cite[Definition 1.1]{MR4698898} says that~$\sfS$ \emph{absorbs singularities of }$\calX_o$.
    \item The category~$\sfA_\sfS$ is the (thick) closure of the pushforward of~$\sfS$ to~$\bd(\calX)$.
    If it is admissible in~$\bd(\calX)$, \cite[Theorem 1.5]{MR4698898} show
    that its complement~$\sfB$ is~$C$-linear, as well as smooth and proper such that
    \begin{equation}\label{eq:fibers-of-complement}
        \sfB_p \simeq \begin{cases}
            {}^\perp \sfS &\text{if }p = o,\\
            \bd(\calX_p) &\text{otherwise}.
        \end{cases}
    \end{equation} 
\end{enumerate}
If additionally, (b) is satisfied for~$\sfS$, then Definition 1.4 of \emph{op.~cit.}
says that~$\sfS$\emph{ provides a deformation absorption of singularities of }$\calX_o$.
\Cref{thm:deformation-absorption-intro} shows that (b) holds in the noncommutative setting
for~$\sfS$ defined in \Cref{thm:absorption-of-singularities-intro}.

Finding~$\sfS \subsetneq \bd(\calX_o)$ which satisfies (a) and (b) is a non-trivial task.
Among other things, Kuznetsov--Shinder thus introduce so-called~$\bbP^{\infty,2}$-objects~$S\in \bd(\calX_o)$.
This is an object~$S$ such that its ring of self-extensions is~$\Ext_{\calX_o}^{\bullet}(S,S) \cong \field[\theta]$ with~$\deg \theta = 2$,
see \Cref{def:p-infty-object}.
By \cite[Theorem 1.8]{MR4698898} a semiorthogonal collection of~$\bbP^{\infty,2}$-objects in~$\bd(\calX_o)$
forms a triangulated subcategory satisfying (a) and (b).
Notably, (b) is satisfied independently of the chosen smoothing~$\calX$ of the singular fiber~$\calX_o$.

Examples of this phenomenon are odd-dimensional varieties with isolated nodal singularities \cite[Theorem 6.1]{MR4698898}.
Generalizing to the notion of compound~$\bbP^{\infty}$-objects, \cite{2402.18513} observed a similar phenomenon
for a projective threefold with non-isolated singularity.
An application of categorical absorption in the context of mirror symmetry is given by \cite{2412.09724}.
Interpreting a hereditary order~$\order{A}$ as a flat family over its central curve~$C$,
we present an example~$(C,\order{A})\to C$ where the role of
the singular fiber is played by
the restriction~$\order{A}(o)$ of~$\order{A}$ to~$o$
in \Cref{sec:categorical-absorption}.

\paragraph{Noncommutative base change.}
In Appendix \ref{appendix:noncommutative-base-change-formula}
we will make sense of the fiber~$\sfB_p$ from \eqref{eq:fibers-of-complement}
in the derived category of a coherent ringed scheme~$(C,\order{A})$
by providing a base change formula
which follows from some modifications of the results in \cite{MR2801403}.

Let $\order{A}$ be a flat $\calO_X$-algebra.
Moreover, let~$f \colon X\to S$ be a flat morphism of schemes.
If~$h\colon T \to S$ is a morphism of schemes,
we set~$X_T = X\times_S T$ and~$\order{A}_T = h_T^*\order{A}$,
where~$h_T\colon X_T\to X$ is the induced morphism.
In \Cref{lem:noncommutative-fibre-product}, we will show
how $(X_T,\order{A}_T)$ is a fiber product.
For its bounded derived category~$\bd(X_T,\order{A}_T)$
we have the following generalization of \cite[Theorem 5.6]{MR2801403}.

\begin{alphatheorem}[\Cref{thm:main-base-change-theorem}]\label{thm:main-base-change-theorem-intro}
    Assume that 
    \begin{equation}
        \bd(X,\order{A}) = \langle\sfA_1,\ldots, \sfA_m \rangle
    \end{equation}
    is an~$S$-linear strong semiorthogonal decomposition
    such that the projection functors have finite cohomological amplitude.
    If~$\order{A}$ has finite global dimension,
    then there is a~$T$-linear semiorthogonal decomposition
    \begin{equation}
        \bd(X_T,\order{A}_T) = \langle \sfA_{1,T},\ldots, \sfA_{m,T}\rangle
    \end{equation}
    compatible with pullback and pushforward.
\end{alphatheorem}
It should be mentioned that \Cref{thm:absorption-of-singularities-intro} and \Cref{thm:deformation-absorption-intro}
are independent from the base change formula in \Cref{thm:main-base-change-theorem-intro}.

\paragraph{Organization of the paper.}
We recall properties of coherent ringed schemes
in \Cref{sec:noncommutative-schemes}.
We then apply this point of view to hereditary orders in \Cref{sec:hereditary-orders}.
In \Cref{sec:categorical-absorption}, we prove \Cref{thm:absorption-of-singularities-intro} and \Cref{thm:deformation-absorption-intro}.
Moreover, we show in \Cref{sec:dictionary-orders-stacks} how our results relate to smooth stacky curves
using the dictionary \cite{MR2018958}.
In the appendix,
we explain how to obtain \Cref{thm:main-base-change-theorem-intro}
from its commutative version \cite{MR2801403}.

\paragraph{Notations and conventions.} Throughout the paper~$\field$ denotes an algebraically closed field of characteristic zero.
Every scheme is assumed to be integral and quasi-projective over~$\field$.
If~$X$ is a scheme, we denote by~$\field(X)$ its function field.

\paragraph{Acknowledgements.}
The author was partially supported by the Luxembourg National Research Fund (PRIDE R-AGR-3528-16-Z).
We are grateful to Pieter Belmans for many helpful discussions and for his feedback on drafts of this paper.

\section{Coherent ringed schemes}\label{sec:noncommutative-schemes}
In this section we define coherent ringed schemes~$(X,\order{A})$
following the terminology of \cite{MR2244264}.
Sometimes they are also referred to as `mild noncommutative schemes', see \cite{2410.01785}.
Their derived category~$\sfD(X,\order{A})$ (and bounded versions of it) were for example studied in
\cite[Appendix D]{MR2238172}, \cite{MR3695056}, \cite[Appendix A]{MR4554471}, \cite{2408.04561,2410.01785}.
\begin{definition}\label{def:noncommutative-schemes}
A \emph{coherent ringed scheme} is
a pair~$(X,\order{A})$ of
a scheme~$X$ over~$\field$ and 
a coherent sheaf of~$\calO_X$-algebras~$\order{A}$.

A \emph{morphism of coherent ringed schemes}~$\mff = (f,f_\alg) \colon (X,\order{A})\to (Y,\order{B})$
consists of 
\begin{itemize}
	\item a morphism of schemes~$f\colon (X,\calO_X)\to (Y,\calO_Y)$.
	\item an~$\calO_X$-algebra morphism~$f_\alg\colon f^*\order{B}\to \order{A}$.
\end{itemize}
\end{definition}

An $\calO_X$-order $\order{A}$,
as defined in \Cref{sec:hereditary-orders},
is a coherent ringed scheme $(X,\order{A})$.
\Cref{lem:morita-equivalence-to-cyclic-algebra} will show that
it is necessary for us to work need in the more
general framework of coherent ringed schemes.

Given~$\mff = (f, f_\alg)\colon (X,\order{A})\to (Y,\order{B})$ and~$\mfg = (g, g_\alg)\colon (Y,\order{B})\to (Z,\order{C})$
we define their \emph{composition}~$\mfh = (h,h_\alg) =\mfg \circ \mff \colon (X,\order{A})\to (Z,\order{C})$ as follows:
\begin{itemize}
    \item on the underlying schemes~$h \colonequals g \circ f$, and
    \item on algebras~$h_\alg \colonequals f_\alg \circ f^*(g_\alg)\colon h^*\order{C}\to f^*\order{B}\to \order{A}$.
\end{itemize}
This allows us to form the category of coherent ringed schemes~$\ncSch_\field$.
Note that every scheme~$X$ is an object in~$\ncSch_\field$ by setting~$\order{A} = \calO_X$.

Following \cite[Definition 10.3]{MR2238172},
we call a morphism~$\mff\colon (X,\order{A})\to (Y,\order{B})$ \emph{strict}
if~$f_\alg = \text{id}_{f^*\order{B}}$.
A morphism~$\mff$ is called an \emph{extension}
if~$X = Y$ and~$f$ is the identity.

\begin{remark}\label{rem:finite-noncommutative-covering}
    Every coherent ringed scheme~$(X,\order{A})$ comes with a \emph{structure morphism}~$\mff\colon (X,\order{A})\to X$.
    This is an extension, where~$f_\alg\colon \calO_X\to \order{A}$ induces the~$\calO_X$-algebra structure on~$\order{A}$.
    Since~$\order{A}$ is a coherent~$\calO_X$-algebra,
    we can view the structure morphism as a \emph{finite noncommutative covering}
    generalizing the equivalence~$\coh(\relspec{X}{\order{A}})\cong \coh(X,\order{A})$
    for commutative~$\order{A}$
    by allowing noncommutative finite-dimensional~$\field$-algebras as fibers, see \cite{MR2356702}.
\end{remark}
\paragraph{Modules on coherent ringed schemes.}
Given a coherent ringed scheme~$(X,\order{A})$, we denote 
\begin{itemize}
    \item by~$\QCoh(X,\order{A})$ the category of right~$\order{A}$-modules which are quasicoherent, and 
    \item by~$\coh(X,\order{A})$ the subcategory of right~$\order{A}$-modules which are coherent as~$\calO_X$-modules. 
\end{itemize}

Similarly to the commutative case,
we have the \emph{sheaf hom-functor}~$\inthom_{\order{A}}(-,-)$,
and the \emph{tensor product of~$\order{A}$-modules}~$-\otimes_\order{A}-$.
Note that in the tensor product,
the second argument needs to be a left~$\order{A}$-module.

Given a morphism~$\mff\colon (X,\order{A})\to (Y,\order{B})$
of coherent ringed schemes,
the \emph{pushforward}~$\mff_*M\colonequals f_*M$
of $M \in \QCoh(X,\order{A})$
carries an induced~$\order{B}$-module structure via 
\begin{equation}
    f_* M \otimes_Y \order{B} \cong f_* (M \otimes_X f^*\order{B}) \xrightarrow{f_*(f_\alg)} f_*(M\otimes \order{A})\xrightarrow{f_*(\mu_{M})} f_*M,
\end{equation}
where~$\mu_M\colon M\otimes_X \order{A}\to M$ is the~$\order{A}$-module structure of~$M$.
The \emph{pullback} of $N\in \QCoh(Y,\order{B})$ along $\mff$ is
given by~$\mff^*N \colonequals f^*N\otimes_{f^*\order{B}}\order{A} \in \QCoh(X,\order{A})$.
\begin{remark}\label{rem:descending-functors-to-coherent-modules}
    Since we assume that~$\order{A}$ is coherent,
    it follows that~$\inthom_\order{A}$,~$\otimes_\order{A}$ and~$\mff^*$ map coherent modules to coherent modules.
    For the pushforward (as in the commutative case),
    one has to additionally require properness of~$f\colon X\to Y$, cf.~\cite[Lemma A.5]{MR4554471}.
\end{remark}

\paragraph{Fiber products of coherent ringed schemes.}
We extend fiber products \cite[Lemma 10.37]{MR2238172} for Azumaya varieties to 
fiber products in~$\ncSch_\field$.
Let~$\mff_1 \colon (X_1,\order{A}_1)\to (X,\order{A})$
and~$\mff_2\colon (X_2,\order{A}_2)\to (X,\order{A})$ be morphisms of coherent ringed schemes.
Moreover, denote by~$p_i\colon X_1\times_X X_2\to X_i$ the two canonical projections.
\begin{equation}\label{eq:general-fiber-product}
    \begin{tikzcd}
          (X_1\times_X X_2,p_j^*\order{A}_j)\arrow[r, "\mfp_2"] \arrow[d, "\mfp_1"]& (X_2,\order{A}_2) \arrow[d, "\mff_2"]\\
          (X_1,\order{A}_1) \arrow[r, "\mff_1"] & (X,\order{A})
    \end{tikzcd}
\end{equation}

\begin{lemma}\label{lem:noncommutative-fibre-product}
    Assume that there exists~$i\in \{1,2\}$ such that the morphism~$\mff_i$ is strict and let~$j\neq i$.
    The coherent ringed scheme~$(X_1\times_X X_2, p_j^*\order{A}_j)$ is a fibre product of~$\mff_1$ and~$\mff_2$ in~$\ncSch_\field$.
    It is unique up to unique isomorphism.
\end{lemma}

\begin{proof}
    By symmetry, we assume without loss of generality that~$\mff_1$ is strict.
    Since~$X,X_1,X_2$ are noetherian,
    the~$\calO_{X_1\times_X X_2}$-algebra~$p_2^*\order{A}_2$ is coherent.
    The morphism~$\mfp_2 = (p_2, \text{id}_{p_2^{*}\order{A}_2})$ is the strict morphism
    induced from the structure morphism~$p_2\colon X_1\times_X X_2 \to X_2$ for schemes.
    The morphism~$\mfp_1 = (p_1, p_2^*(f_{2,\alg}))$ is given
    by the structure morphism~$p_1\colon X_1\times_X X_2\to X_1$ for schemes and
    the~$\calO_{X_1\times_X X_2}$-algebra homomorphism~$p_2^*(f_{2,\alg})\colon p_{1}^*\order{A}_1 \to p_2^*\order{A}_2$.
    Note that~$p_{1}^*\order{A}_1 \cong p_2^*f_2^*\order{A}$ because~$\mff_1$ is strict.

    It is straightforward to verify the universal property
    using the universal property for the fiber product of the underlying
    (commutative schemes).
\end{proof}
\begin{remark}\label{rem:base-change-structure-morphism}
    Note that a morphism~$h\colon T\to S$ of schemes is always strict.
    Hence we can base change a coherent ringed scheme~$\mff = (f, f_\alg)\colon(X,\order{A})\to S$ over a commutative base~$S$
    along~$h\colon T\to S$, to obtain a coherent ringed scheme~$(X_T, \order{A}_T)$ over~$T$,
    where~$X_T = X\times_S T$ and~$\order{A}_T = p_2^*\order{A}$.
    We will return to this perspective in \Cref{sec:deformation-absorption-of-singularities}
    for the deformation absorption in \Cref{thm:deformation-absorption}.
\end{remark}

\paragraph{K-injectives and locally projectives.}
The category~$\QCoh(X,\order{A})$ is a Grothendieck abelian category,
and therefore every cochain complex~$M^\bullet$ of quasicoherent~$\order{A}$-modules admits a K-injective resolution, see \cite[Lemma A.2]{MR4554471}.
The notion of locally free modules has to be replaced by locally projective modules.

\begin{definition}\label{def:locally-projective-modules}
A coherent~$\order{A}$-module~$P\in \coh(X,\order{A})$ is \emph{locally projective} if there exists a (Zariski-)open covering~$X = \bigcup_{i\in I} U_i$ such that~$P\vert_{U_i}$ is a finitely generated projective~$\order{A}\vert_{U_i}$-module for all~$i\in I$.
\end{definition}

By \cite[Corollary 3.23]{MR393100},
this definition agrees with \cite[Definition 3.6]{MR3695056}.
\begin{lemma}\label{lem:locllay-projective-is-local}
    Let~$(X,\order{A})$ be a coherent ringed scheme and~$P\in \coh(X,\order{A})$ a coherent~$\order{A}$-module. Then the following are equivalent:
    \begin{enumerate}[label = \roman*)]
        \item The~$\order{A}$-module~$P$ is locally projective.
        \item For every~$p\in X$, the~$\order{A}_{p}$-module~$P_p$ is projective.
    \end{enumerate}
\end{lemma}
Under the (standing) assumption that~$X$ is quasi-projective,
\cite[Proposition 3.7]{MR3695056} show that~$\coh(X,\order{A})$ \emph{admits enough locally projectives},
i.e.~for every coherent~$\order{A}$-module~$M$ there exists a locally projective~$\order{A}$-module~$P$
and an~$\order{A}$-module epimorphism~$P\twoheadrightarrow M$.

\paragraph{The derived category of a coherent ringed scheme.}
Let~$(X,\order{A})$ be a coherent ringed scheme.
We denote by~$\sfD(X,\order{A}) \colonequals \sfD(\QCoh(X,\order{A}))$
the \emph{unbounded derived category of quasicoherent~$\order{A}$-modules}.
By~$\sfD^*(X,\order{A})$, for~$* \in \{+, -, \sfb\}$,
we denote the bounded below, bounded above, resp.~bounded
derived category of quasicoherent~$\order{A}$-modules.

Mostly, we are interested in
the \emph{bounded derived category}~$\bd(X,\order{A}) \colonequals \sfD^{\sfb}_\coh(\QCoh(X,\order{A}))$.
Since~$\order{A}$ is coherent, \cite[Lemma A.4]{MR4554471} provides the useful equivalence
\begin{equation}\label{eq:coherent-cohomology-and-bounded-derived-category}
    \bd(X,\order{A}) \simeq \bd(\coh(X,\order{A})).
\end{equation}
From time to time, we will have to make the distinction between $\bd(X,\order{A})$
and the \emph{category of perfect complexes}~$\sfD^\perf(X,\order{A})$,
which is the full triangulated subcategory of~$\sfD(X,\order{A})$
given by cochain complexes that are locally quasi-isomorphic to bounded complexes of locally projective~$\order{A}$-modules.
For a hereditary order $\order{A}$ on a curve $C$,
we have an equality $\sfD^\perf(C,\order{A})= \bd(C,\order{A})$
by \Cref{lem:D-perf-globally}.

\begin{remark}
    It is a straightforward, but important observation, cf.~\cite[Example 3.7]{2408.04561},
that the pushforward~$\mff_* \colon \sfD(X,\order{A})\to \sfD(X)$ of the structure morphism~$\mff \colon (X,\order{A})\to X$
is exact as exactness does not depend on the module structure.
\end{remark}

\paragraph{Semiorthogonal decompositions linear over the base.}  
Let~$\sfT = \bd(X,\order{A})$,
and consider a morphism of schemes~$\mff \colon (X,\order{A})\to S$.
Recall from \cite[Section 2.6]{MR2238172}, \cite[\S\S 2,3]{MR3948688}
that a triangulated subcategory~$\sfA\subset \sfT$
is called~$S$\emph{-linear} if 
\begin{equation}
    M\otimes_X^{\bfL} \bfL f^*\calF \in \sfA \quad \text{for all }M \in \sfA, \calF \in \sfD^\perf(S).
\end{equation}

\begin{remark}\label{rem:extended-S-linearity}
    If~$\sfT = \sfD^{-}(X,\order{A})$ or~$\sfT = \sfD(X,\order{A})$,
    and~$\sfA\subset \sfT$ is~$S$-linear for~$\calF \in \sfD^\perf(S)$,
    then~$S$-linearity automatically holds for~$\sfD^-(S)$, resp.~$\sfD(S)$.
    This follows from \cite[Lemma 4.5]{MR2801403}.
\end{remark}
Let~$\sfT$ be an~$S$-linear triangulated category.
If~$\sfA\subset \sfT$ is a full triangulated subcategory,
the \emph{right orthogonal}~$\sfA^\perp$ (resp.~\emph{left orthogonal}~${}^\perp\sfA$) of~$\sfA$
is defined as the full subcategory containing all objects
$T\in \sfT$ such that~$\Hom_{\sfT}(A,T[i]) = 0$ (resp.~$\Hom_{\sfT}(T,A[i]) = 0$) for all~$A\in \sfA$ and~$i\in \bbZ$.

The next lemma is a straightforward generalization of \cite[Lemma 2.36]{MR2238172}.
\begin{lemma}\label{lem:linearity-of-complement}
    Let~$\mff \colon (X,\order{A})\to S$ be a morphism of schemes,
    and~$\sfA \subset \bd(X,\order{A})$ be an~$S$-linear category 
    such that~$\sfA$ is~$S$-linear.
    Then~${}^{\perp}\sfA$ and~$\sfA^{\perp}$ are~$S$-linear as well.
\end{lemma}

Recall that an ($S$-linear) triangulated subcategory~$\sfA$ of~$\sfT$ is called \emph{right admissible} (resp.~\emph{left admissible})
if its embedding functor~$\alpha \colon \sfA \to \sfT$ admits a right adjoint~$\alpha^!$ (resp.~a left adjoint~$\alpha^*$).
We say that~$\sfA$ is \emph{admissible} if it is both, left and right admissible.

A sequence of ($S$-linear) triangulated subcategories~$\sfA_1,\ldots, \sfA_m$ is called
an \emph{($S$-linear) semiorthogonal decomposition} of~$\sfT$ if
\begin{enumerate}[label = \roman*)]
    \item for every~$i>j$ one has~$\sfA_j \subset \sfA_i^\perp$,
    \item the smallest triangulated subcategory containing all~$\sfA_i$ is~$\sfT$.
\end{enumerate}
In this case we write~$\sfT = \langle \sfA_1,\ldots, \sfA_m\rangle$ and
every object~$T\in \sfT$ can be decomposed into distinguished triangles
\begin{equation}\label{eq:split-triangles-from-sod}
    T_\ell \to T_{\ell -1} \to A_\ell\to T_{\ell}[1]\quad\text{for }1\le\ell \le m,
\end{equation}
such that~$T_m = 0$,~$T_0 = T$, and~$A_\ell \in \sfA_\ell$.

Admissibility of components of a semiorthogonal decomposition is not automatic.
Therefore \cite[Definition 2.6]{MR2801403} introduced the notion of strong semiorthogonal decompositions.
\begin{definition}
    We call a semiorthogonal decomposition~$\sfT = \langle \sfA_1,\ldots, \sfA_m\rangle$ \emph{strong}
    if the component~$\sfA_k$ is admissible in~$\langle \sfA_k,\ldots, \sfA_m\rangle$. 
\end{definition}
We finish this section with a relative criterion
for~$S$-linear semiorthogonality generalizing \cite[Lemma 2.7]{MR2801403}
to coherent ringed schemes.
\begin{lemma}\label{lem:rel-criterion-semiorthogonality}
    Let~$\mff\colon (X,\order{A})\to S$ be a morphism and~$\sfA,\sfB\subseteq \sfD(X,\order{A})$ be~$S$-linear admissible subcategories.
    Then~$\sfA\subseteq \sfB^{\perp}$ if and only if~$\bfR f_*\bfR\inthom_{\order{A}}(\sfB,\sfA) = 0$. 
\end{lemma}
\begin{proof}
    We start with the assumption that~$\sfA \subset \sfB^\perp$.
    By the proof of \cite[Lemma 2.7]{MR2801403}
    it suffices to show that~$\bfR \Hom_S(P,\bfR f_* \bfR \inthom_{\order{A}}(N,M)) = 0$
    for all~$P \in \sfD^\perf(S)$,~$N \in \sfB$ and~$M \in \sfA$.
    Using the adjunction~$\bfL f^* \dashv \bfR f_*$ and the tensor-hom adjunction, one obtains
    \begin{equation*}
        \bfR \Hom_S(P,\bfR f_* \bfR \inthom_{\order{A}}(N,M)) \cong \bfR \Hom_S(\bfL f^* P, \bfR \inthom_{\order{A}}(N,M))
        \cong \bfR\Hom_{\order{A}}(\bfL f^* P \otimes^\bfL_X N, M) = 0.
    \end{equation*}
    The vanishing in the last step follows from the~$S$-linearity of~$\sfB$.
    For the converse, note that 
    \begin{equation}
        \bfR\Hom_S(\calO_S,\bfR f_* \bfR \inthom_{\order{A}}(N,M)) \cong \bfR\Hom_X(\calO_X,\bfR \inthom_{\order{A}}(N,M))
        \cong \bfR\Hom_{\order{A}}(N,M)
    \end{equation}
    by adjunction and the fact that the pushforward along the structure morphism~$(X,\order{A})\to X$ is exact.
\end{proof}

\section{Hereditary orders}\label{sec:hereditary-orders}
Let~$X$ be an (integral) scheme.
An~$\calO_X$\emph{-order} is a coherent~$\calO_X$-algebra~$\order{A}$ which
is torsion-free as an~$\calO_X$-module
such that~$\order{A}\otimes_X \field(X)$ is a central simple~$\field(X)$-algebra.
Consequently, $(X,\order{A})$ is a coherent ringed scheme.

By definition of an order $\order{A}$ there is a maximal open dense subset~$U\subset X$
such that the restriction~$\order{A}\vert_U$ is Azumaya.
The complement~$\Delta_\order{A} = X\setminus U$ is the \emph{ramification locus of}$\order{A}$.

We say that an~$\calO_X$-order~$\order{B}$ is an \emph{overorder of }$\order{A}$
if~$\order{A}\subseteq\order{B} \subseteq \order{A}\otimes_X\field(X)$.
The order~$\order{A}$ is \emph{maximal} if there are no proper overorders of~$\order{A}$.

Throughout this section,
let $C$ be a smooth quasi-projective curve over~$\field$,
and~$\order{A}$ a hereditary~$\calO_C$-order with ramification locus~$\Delta_\order{A}$.
Denote by~$\mff \colon (C,\order{A})\to C$ the structure morphism.
It is shown in \cite[Corollary 3.24, Theorem 40.5]{MR393100}
that being a hereditary order is a local property.
This means that $\order{A}$ is a hereditary $\calO_C$-order if and only if
the localization~$\order{A}_p$ is a hereditary~$\calO_{C,p}$-order for every point~$p\in C$.

In the following, let~$R = \calO_{C,p}$ be
the local ring at~$p\in C$
and~$\mfm \unlhd R$ be the maximal ideal of~$R$.
Since $R$ is the unique maximal $R$-order in $\field(C)$ by normality of~$C$,
it follows from \cite[Theorem 39.14]{MR393100}
that there is an isomorphism of~$R$-algebras 
\begin{equation}\label{eq:classification-hereditary-order}
    \order{A}_p \cong \begin{pmatrix}
        R & R & \ldots & R \\
        \mfm & R &\ldots & R \\
        \vdots & & \ddots &\vdots\\
        \mfm & \mfm & \ldots & R
    \end{pmatrix}^{(n_1,\ldots, n_r)}\quad \subseteq \Mat_n(R),
\end{equation}  
where~$n^2 = \rk \order{A}$,
$n_1+\ldots+n_r = n$,
and the superscript~$(n_1, \ldots, n_r)$ indicates that
the~$(i,j)$-th coordinate of the matrix on the right-hand side
is to be read as an~$n_i\times n_j$- matrix with entries in~$\mfm$, or~$R$ respectively.
We call~$(n_1,\ldots,n_r)\in \bbN^r$ the \emph{ramification data} of~$\order{A}$ at~$p \in C$.
If~$\order{A}$ is an~$\calO_C$-order
such that the ramification data at~$p\in C$
satisfies~$r > 1$,
then~$\order{A}$ is \emph{ramified at }$p$ with \emph{ramification index}~$r$.

By \cite[Theorem 39.23]{MR393100},
the indecomposable projective~$\order{A}_p$-modules are given (up to isomorphism)
by the rows 
\begin{equation}\label{eq:locally-indec-projectives}
    L_p^{(j)} = E_{\alpha_{j},\alpha_{j}}\order{A}_p,\quad j = 1,\ldots, r,
\end{equation}
where~$\alpha_j = n_1+\ldots+n_{j}$, and~$E_{\alpha,\beta} \in \Mat_n(R)$ denote the elementary matrices.

\paragraph{The fibers of hereditary orders.}
Let $p\in C$.
We describe the fiber~$\order{A}(p)\colonequals \order{A}_p\otimes_R\field(p)$ over~$p$.
In the language of \Cref{rem:base-change-structure-morphism}
this is the base change of
the structure morphism~$\mff\colon (C,\order{A})\to C$
along the closed immersion~$i_p\colon \Spec \field(p)\to C.$ 

Let~$Q_r$ be the cyclic quiver
with~$r$ vertices~$Q_0 = \{1,\ldots, r\}$
and~$r$ arrows~$Q_1 = \{\mu_{i,i+1}\colon i \to i+1\}$,
where here and in the following the numbering has to be understood modulo~$r$.
The quiver~$Q_r$ can be depicted as follows:
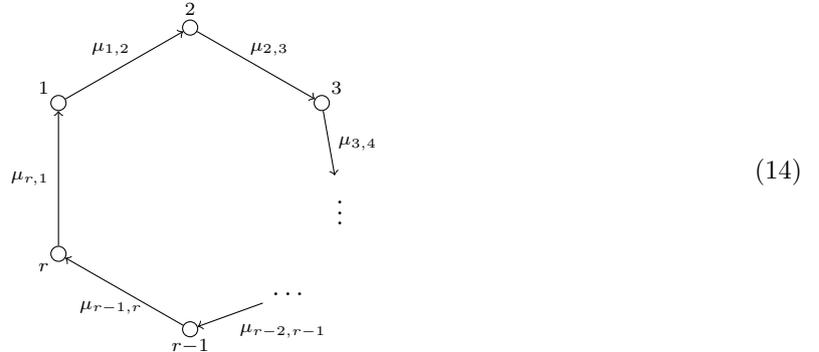
\begin{equation}\label{eq:cyclic-quiver}
    \begin{tikzpicture}[baseline,vertex/.style={draw, circle, inner sep=0pt, text width=2mm}, scale=2.5]
        \node[vertex] (a) at (150:.8cm) {};
        \node[vertex] (b) at (90:.8cm)  {};
        \node (b1) at (90:.9cm) {};
        \node[vertex] (c) at (30:.8cm)  {};
        \node         (d) at (350:.8cm) {$\vdots$};
        \node         (e) at (310:.8cm) {$\ldots$};
        \node[vertex] (f) at (270:.8cm) {};
        \node[vertex] (g) at (210:.8cm) {};
        \node[above left] at (a) {$\scriptstyle 1$};
        \node       at (b1) {$\scriptstyle 2$};
        \node[above right]  at (c) {$\scriptstyle 3$};
        \node[below]       at (f) {$\scriptstyle r-1$};
        \node[below left] at (g) {$\scriptstyle r$};
        \draw[-{Classical TikZ Rightarrow[]}] (a) -- node [above] {$\scriptstyle\mu_{1,2}$\quad\ } (b);
        \draw[-{Classical TikZ Rightarrow[]}] (b) -- node [above] {\quad$\scriptstyle\mu_{2,3}$} (c);
        \draw[-{Classical TikZ Rightarrow[]}] (c) -- node [right] {$\scriptstyle\mu_{3,4}$} (d);
        \draw[-{Classical TikZ Rightarrow[]}] (e) -- node [below right] {$\scriptstyle\mu_{r-2,r-1}$} (f);
        \draw[-{Classical TikZ Rightarrow[]}] (f) -- node [below] {$\scriptstyle\mu_{r-1,r}\quad$} (g);
        \draw[-{Classical TikZ Rightarrow[]}] (g) -- node [left] {$\scriptstyle\mu_{r,1}$} (a);
    \end{tikzpicture}
\end{equation}
Denote by~$\mu_{[i,j]}\colon i \to i+1 \to \ldots \to j-1 \to j$
the shortest path of positive length from~$i$ to~$j$ and let
\begin{equation}
    I = (\rad\field Q_r)^{n-1} \triangleleft \field Q_r
\end{equation}
be the admissible ideal generated by all cycles~$\mu_{[i,i]}$.
The quotient
\begin{equation}\label{eq:fibre-over-closed-point}
    \Lambda_r = \field Q_r/I
\end{equation}
is an~$r^2$-dimensional~$\field$-algebra of infinite global dimension.

\begin{lemma}\label{lem:morita-equivalence-to-cyclic-algebra}
    Let~$(C,\order{A})$ be a hereditary order which is ramified at~$p\in C$ with ramification index~$r$.
    Then the~$\field$-algebra~$\order{A}(p) = \order{A}_p\otimes_C \field(p)$ is Morita equivalent to~$\Lambda_r$.
\end{lemma}

\begin{proof}
By \cite[Theorem 7.6]{MR2018958}, we find that~$\order{A}_p$ is
Morita equivalent to an~$R$-algebra~$\Gamma \subset \Mat_r(R)$
as in \eqref{eq:classification-hereditary-order}
with~$n_1 = \ldots = n_r = 1$.
The algebra~$\Gamma\otimes_R\field(p)$ is isomorphic to~$\Lambda_r$
by \cite[Theorem 3.1]{MR2029238},
and hence~$\order{A}(p)$ is Morita equivalent to~$\Lambda_r$.

We sketch a different proof of why~$\Gamma\otimes_R\field$ is isomorphic to~$\Lambda_r$
using the algorithmic characterization of finite-dimensional basic~$\field(p)$-algebras
as quotients of path algebras of a quiver by an admissible ideal from \cite[Section II.3]{MR2197389}.
By \cite[Section 2]{MR657429},
there is a~$\field$-basis~$\{e_{ij}\}_{1\le i,j \le r}$ of~$\Gamma\otimes_R\field(p)$ such that 
\begin{equation}\label{eq:multiplication-rule-fiber}
	e_{ij}\cdot e_{jk} = \begin{cases}
	e_{ik} &\text{if } i\le j\le k\text{ or } j\le k < i \text{ or }k< i\le j,\\
	0 &\text{otherwise.}
	\end{cases}
\end{equation}
Moreover, one has~$e_{ij}\cdot e_{j^\prime k} = 0$ if~$j\neq j^\prime$.
The algebra~$\Gamma\otimes_R\field(p)$ is basic and connected,
with~$e_{11},\ldots, e_{rr}$ a complete set of primitive orthogonal idempotents.

Define a~$\field$-algebra homomorphism ~$\varphi\colon \field Q_r\to \Gamma\otimes_R\field(p)$
by sending the lazy path~$e_i$,
associated with the vertex~$i\in Q_0$,
to the idempotent~$e_{ii}$,
and the arrow~$\mu_{i,i+1}\in Q_1$ to~$e_{i,i+1}$.
It follows from \eqref{eq:multiplication-rule-fiber} that this is a surjective~$\field$-algebra homomorphism.
For dimension reasons and from the relations~$e_{i,i+1}\cdot \ldots\cdot e_{i-1,i}=0$ for each~$i = 1,\ldots r$,
it follows that~$\varphi$ induces an isomorphism~$\Lambda_r\cong \Gamma\otimes_R\field(p)$. 
\end{proof}

We can use the representation theory of~$\field Q_r/I$ to characterize simple~$\Lambda_r$-modules.
Since there are no non-zero oriented cycles, thanks to the relations,
the simple~$\Lambda_r$-modules up to isomorphism are given by~$S_1,\ldots, S_r$ with
\begin{equation}\label{eq:definition-simple-module}
    (S_i)_j = \begin{cases}
        \field &\text{if } j = i,\\
        0 &\text{if }j \in Q_0\setminus\{i\}.
    \end{cases}
\end{equation}

\begin{remark}
The number of isomorphism classes of simple~$\order{A}(p)$-modules does not change under the
Morita equivalence from \Cref{lem:morita-equivalence-to-cyclic-algebra}.
It is straightforward that the simple~$\order{A}(p)$-module corresponding to~$S_i$ consists of a single~$n_i$-dimensional
non-trivial~$\field$-vector space
at the vertex~$i\in Q_0$.
In the remainder we will only be interested in properties of the simple~$\Lambda_r$-modules
which are preserved under the Morita equivalence. 
In particular \Cref{thm:absorption-of-singularities} and \Cref{thm:deformation-absorption} do only depend
on~$\order{A}(p)$ up to Morita equivalence.
\end{remark}

\paragraph{Maximal overorders.}
Denote by~$\Delta_\order{A} = \{p_1,\ldots, p_m\}\subset C$ the ramification locus of~$\order{A}$ and
let~$r_i>1$ be the corresponding ramification index of~$\order{A}$ at~$p_i\in C$.

The inclusion of every maximal overorder~$j_{\order{B},\alg}\colon\order{A}\hookrightarrow\order{B}$
gives rise to an extension
\begin{equation}
    \mfj_{\order{B}}= (\text{id}_{\order{B}},j_{\order{B},\alg})\colon (C,\order{B})\to (C,\order{A}).
\end{equation}
It was shown in \cite[Theorem 3.1]{MR219527} that
the pushforward~$\mfj_{\order{B},*}\order{B}$ is
a locally projective (left and right)~$\order{A}$-module.
We study the shape of the locally projective~$\order{A}$-modules more closely
with respect to the indecomposable projective~$\order{A}_p$-modules~$L_p^{(j)}$ for~$j = 1,\ldots,r$
from \eqref{eq:locally-indec-projectives}.

\begin{definition}\label{def:type-of-maximal-overorder}
    Let~$p\in \Delta_\order{A}$ such that~$\order{A}$ has ramification index~$r$ at~$p$.
    A nonzero locally projective~$\order{A}$-module~$P$ is called \emph{purely of type~$j$ at~$p$} if
    \begin{equation}
        P_p \cong {L_p^{(j)}}^{\oplus k}\quad \text{for some }k\in \bbN.
    \end{equation}
    If the type~$j$ is not specified, we say that~$M$ is \emph{purely of one type at~$p$}.
\end{definition}

As shown in \cite[Theorem 40.10]{MR393100},
hereditary orders can be characterized by the maximal orders containing them.
Using the classification of the indecomposable projective~$\order{A}_p$-modules,
one can construct all the maximal orders containing~$\order{A}$ explicitly.

\begin{proposition}\label{prop:classification-of-overorders}
    Let~$(C,\order{A})$ be a hereditary order with ramification data as above.
    \begin{enumerate}[label = \roman*)]
        \item There are precisely~$r_1\cdot \ldots \cdot r_m$ maximal overorders of~$\order{A}$.
        \item Every maximal overorder~$\order{B}$ of~$\order{A}$ is locally projective as a left and a right~$\order{A}$-module.
        \item Every maximal overorder~$\order{B}\supseteq \order{A}$ is purely of one type at each ramification point.
        Moreover, a maximal overorder is uniquely determined by its types at the ramification points.
    \end{enumerate}
\end{proposition}
\begin{proof}
    The first two parts can be found in \cite[Theorem 39.23 and Theorem 40.8]{MR393100}.
    Let us indicate how to prove these two statements for an algebraically closed field
    using the classification \eqref{eq:locally-indec-projectives}
    of the indecomposable projective~$\order{A}_p$-modules for~$p\in C$.

    By Tsen's theorem,~$\order{A}\otimes_C\field(C)\cong \End_{\field(C)}(V)$, 
    for some~$n$-dimensional vector space~$V$.
    Let~$\order{B}\supset\order{A}$ be a maximal overorder.
    Then~$\order{B}_p\supset \order{A}_p$ is a maximal overorder.
    From the classification of maximal orders in discrete valuation rings (see Corollary 17.4 of \emph{op.~cit.}),
    it follows that there exists an~$\calO_{C,p}$-lattice~$L_p$ in~$V$
    such that~$\order{B}_p = \End_{\calO_{C,p}}(L_p)$ and
   ~$L_p$ has the structure of a right~$\order{A}_p$-module
    from the induced~$\order{B}_p$-module structure on~$L_p$.
    It must be indecomposable as~$\order{A}_p$-module, because its endomorphism ring 
    is generically central simple.
    Therefore~$L_p$ is an indecomposable projective~$\order{A}_p$-module,
    e.g.~by \cite[Theorem 10.6]{MR393100} and the fact that~$\order{A}_p$ is hereditary.

    Now~$\order{B}\subset \End_{\field(C)}(V)$ is coherent and torsion-free, and therefore uniquely determined
    by the localization at every point~$p\in C$, see \cite[Proposition 6.4]{MR4385131}.
    For a ramified point~$p_i\in \Delta_i$
    we have~$r_i$ choices of~$L_{p_i}= L_{p_i}^{(j)}$,~$j= 1,\ldots, r_i$,
    hence~$r_i$ different possible specializations of~$\order{B}$ at~$p_i$,
    and we conclude (i).

    For the second part,
    we pass to the local ring at a point $p\in C$ using \Cref{lem:locllay-projective-is-local}.
    From the construction of~$\order{B}$ above,
    it is clear that~$\order{B}$ is locally projective as a right~$\order{A}$-module.
    The isomorphism as left modules follows similarly using
    that~$\Hom_{\calO_{C,p}}(L_{p_i}^{(j)},\calO_{C,p})$ describes indecomposable projective left modules. 

    The third part can be calculated locally as well.
    We have at~$p_i \in \Delta_\order{A}$ that  
    \begin{equation}
        \order{B}_{p_i} \cong \End_{\calO_{C,p_i}}(L_{p_i}^{(j)}) \cong
        \Hom_{\calO_{C,p_i}}(\calO_{C,p_i}^{\oplus n},L_{p_i}^{(j)})
        \cong {L_{p_i}^{(j)}}^{\oplus n}.
    \end{equation}
    The isomorphism as left~$\order{A}_{p_i}$-modules is given similarly.
    Uniqueness follows from the local modification theorem \cite[Proposition 6.4]{MR4385131}.
\end{proof}

The explicit description of~$\order{B}$
allows us to draw some powerful conclusions about~$\mfj_{\order{B}}$.
\begin{lemma}\label{lem:properties-of-morphisms-by-maximal-overorders}
    Let~$\mfj_\order{B}\colon (C,\order{B})\to (C,\order{A})$
    be the extension corresponding to a maximal overorder~$\order{B}\supset \order{A}$.
    Then the following hold.
    \begin{enumerate}[label = \roman*)]
        \item The pullback~$\mfj_\order{B}^* \colon \QCoh(C,\order{A})\to \QCoh(C,\order{B})$ is exact.
        \item The pushforward~$\mfj_\order{B,*}\colon\QCoh(C,\order{B})\to \QCoh(C,\order{A})$ is exact and fully faithful.
    \end{enumerate}
\end{lemma}

\begin{proof}
    The pushforward~$\mfj_{\order{B},*}$ is the restriction of scalars to~$\order{A}$, forgetting the~$\order{B}$-module structure.
    Therefore, it does not affect exactness of a sequence, and hence~$\mfj_{\order{B},*}$ is exact.

    Since~$\mfj_{\order{B}}$ is an extension,
    the pullback is given by~$\mfj_{\order{B}}^* = -\otimes_{\order{A}}\order{B}$.
    By \Cref{prop:classification-of-overorders},~$\order{B}$ is locally projective,
    hence~$\mfj_{\order{B}}^*$ is exact.

    For the second statement it suffices to show that~$\mfj_{\order{B}}^*\mfj_{\order{B},*}M \cong M$ for all~$M\in \coh(C,\order{B})$.
    As~$\mfj_{\order{B}}^*\mfj_{\order{B},*}M = M\otimes_{\order{A}}\order{B}$, the~$\order{B}$-module structure leads to a homomorphism 
   ~$M\otimes_{\order{A}}\order{B}\to M$ with inverse (locally) given by~$m \mapsto m \otimes 1$. 
\end{proof}

\section{Deformation absorption applied to hereditary orders}\label{sec:categorical-absorption}
We start by recalling in \Cref{sec:background} the notion of
absorption of singularities and deformation absorption of singularities from \cite{MR4698898}
for a flat family of varieties over a curve.
A crucial role is played by~$\bbP^{\infty,2}$-objects,
which can only exist in the bounded derived category of a singular variety~$X$.

In \Cref{sec:absorption-of-singularities} and \Cref{sec:deformation-absorption-of-singularities}
we apply this theory to hereditary orders,
where the role of the singular fiber is played
by the restriction of the order to a ramified point.
This yields a semiorthogonal decomposition of
the bounded derived category of a hereditary order.
In \Cref{sec:dictionary-orders-stacks} we compare this result to the
semiorthogonal decomposition \cite[Theorem 4.7]{MR3573964} for stacky curves.

\subsection{Background}\label{sec:background}
\paragraph{$\bbP^{\infty}$-objects.}
Let~$\sfT$ be a triangulated category and~$q\in \bbZ_{>0}$. We recall \cite[Definition 2.6]{MR4698898}.
\begin{definition}\label{def:p-infty-object}
    An object~$S\in \sfT$ is a~$\bbP^{\infty,q}$\emph{-object} if
    \begin{enumerate}[label = \roman*)]
        \item there is a~$\field$-algebra isomorphism~$\Ext_{\sfT}^\bullet(S,S) \cong \field[\theta]$
        with~$\deg \theta = q$, and
        \item the induced map~$\theta\colon S\to S[q]$ satisfies~$\hocolim (S\to S[q] \to S[2q]\to \ldots) = 0$
        in a cocomplete category~$\widehat{\sfT}$ containing~$\sfT$.
    \end{enumerate}
\end{definition}

Similarly to \cite[Remark 2.7]{MR4698898} for the bounded derived category of a projective variety
the definition simplifies for finite-dimensional~$\field$-algebras.
\begin{remark}\label{rem:easier-p-infty-object}
    Let~$A$ be a finite-dimensional~$\field$-algebra.
    An object~$S\in \bd(A)$ is a~$\bbP^{\infty,q}$-object
    if and only if~$\Ext_{\sfT}^\bullet(S,S) \cong \field[\theta]$.

    Indeed, since~$\sfD(A)\supset \bd(A)$ is cocomplete
    with a compact generator~$A$ which satisfies for every~$M\in \bd(A)$
    that~$\Ext_A^\bullet(A,M)\cong \tH^\bullet(M)$ is bounded above, 
    we can apply \cite[Lemma 2.3]{MR4698898}.
\end{remark}
Every~$\bbP^{\infty,q}$-object comes with the distinguished triangle
\begin{equation}\label{eq:canonical-self-extension}
    M\to S \xrightarrow{\theta} S[q] \to M[1]
\end{equation}
induced by the non-trivial morphism~$\theta\in \Hom_{\sfT}(S,S[q]) = \Ext_\sfT^q(S,S)$.
Following \cite[Definition 2.8]{MR4698898}, we call this triangle the \emph{canonical self-extension} of~$S$.

\begin{remark}
    The object~$M$ can be used to detect whether
    the triangulated subcategory~$\langle S\rangle\subset T$
    generated by~$S$
    is admissible by \cite[Lemma 2.10]{MR4698898}.
    For right (resp.~left) admissibility one needs~$M$ to be \emph{homologically left} (resp.~\emph{right})\emph{ finite-dimensional},
    that is,~$\Ext_\sfT^\bullet(M,N)$, (resp.~$\Ext_\sfT^\bullet(N,M)$) is finite-dimensional for all~$N\in \sfT$,
    see \cite[\S 4.1]{MR4868123} for a definition.
\end{remark}

\paragraph{Overview and definitions.}
A (sequence of)~$\bbP^{\infty,q}$-object(s) detects the difference
between smooth and singular varieties resp.~categories of finite
and infinite global dimension.
In special cases, they generate a proper subcategory~$\sfS\subset \sfT$ 
which absorbs the singularities in the sense of \cite[Definition 1.1]{MR4698898}
for projective varieties.
We present an analogous definition for~$\sfT = \bd(\Lambda)$,
where~$\Lambda$ is a finite-dimensional~$\field$-algebra. 
\begin{definition}\label{def:absorption-singularities}
    A triangulated subcategory~$\sfS \subset \sfT$ absorbs singularities of~$\sfT$
    if~$\sfS$ is admissible and its complements~$\sfS^{\perp}$
    as well as~${}^\perp\sfS$ are smooth and proper.
\end{definition}
\begin{remark}\label{rem:absorption-singularities}
    Note that by \cite[Theorem A]{MR4088795}~$\sfT =\bd(\Lambda)$ is always smooth,
    but not proper unless~$\Lambda$ has finite global dimension.
    This is the same phenomenon as for~$\bd(X)$, where~$X$ is a projective variety.
\end{remark}

We propose the following definition for a noncommutative smoothing.
\begin{definition}\label{def:nc-smoothing}
    Assume that~$\mff\colon (\calX, \order{B})\to C$ is a morphism of coherent ringed schemes
    such that~$C$ is a smooth pointed curve with fixed (closed) point~$o\in C$.
    If the fiber over~$o\in C$ is 
    \begin{equation}
        (\calX \times_C \Spec\field(o), \order{B}_{\Spec\field(o)}) = (X,\order{A})
    \end{equation}
    we say that~$\mff\colon (\calX, \order{B})\to C$ is a \emph{smoothing for }$(X,\order{A})$
    if the following conditions are satisfied:
    \begin{itemize}
        \item for each~$p\in C\setminus\{o\}$ the fiber~$(\calX\vert_{\Spec\field(p)}, \order{B}_{\Spec\field(p)})$
        is a smooth scheme with an Azumaya algebra~$\order{B}_{\Spec\field(p)}$,
        \item the total space~$\calX$ is smooth and~$f$ is flat, and
        \item the~$\calO_X$-algebra~$\order{B}$ flat as an~$\calO_X$-module, and~$\gldim \order{B} = \dim \calX$.
    \end{itemize}
\end{definition}

Kuznetsov--Shinder \cite{MR4698898} provided powerful results for the derived category
of such a smoothing, when~$\order{A}= \calO_X$ and~$\order{B}= \calO_\calX$.
We recall some of their results, stressing the link to~$\bbP^{\infty,2}$-objects.

If a singular projective variety~$X$ admits a semiorthogonal collection~$S_1,\ldots, S_r$
of~$\bbP^{\infty,2}$-objects such that the triangulated subcategory~$\sfS\subset \bd(X)$ generated
by these objects absorbs singularities of~$X$, then by \cite[Theorem 1.8]{MR4698898}:
\begin{itemize}
    \item the pushforward of~$S_1,\ldots, S_r$ to any smoothing~$\calX \to B$
    defines a collection of exceptional objects in~$\bd(\calX)$, and
    \item  the triangulated subcategory in~$\bd(\calX)$ 
    \emph{provides a deformation absorption of singularities of }$X$\emph{ with respect to any smoothing }$\calX\to B$,
    i.e. it is admissible in~$\bd(\calX)$.
\end{itemize} 
In particular the second point implies by Theorem 1.5 of \emph{op.~cit.} that
there is a~$B$-linear semiorthogonal decomposition 
\begin{equation}
    \bd(\calX) = \langle \sfA_S, \sfB\rangle
\end{equation}
where~$\sfA_\sfS$ is the triangulated subcategory generated by the pushforward of~$\sfS$,
and~$\sfB$ is smooth and proper such that the base change satisfies
\begin{equation}
    \sfB_p \simeq \begin{cases}
        {}^\perp \sfS & \text{if }p= o,\\
        \bd(\calX_p) & \text{else}.
    \end{cases}
\end{equation}
In \Cref{thm:deformation-absorption} we provide the last two results in the special case where a hereditary order 
is viewed as the smoothing of the finite-dimensional algebra~$\Lambda_r$ defined in \eqref{eq:fibre-over-closed-point}.

\subsection{Absorption of singularities for the fiber over a ramified point}\label{sec:absorption-of-singularities}
Let~$\order{A}$ be a hereditary order over a curve~$C$.
Let~$o\in C$ be a ramified point with ramification index~$r$.
From \Cref{lem:morita-equivalence-to-cyclic-algebra} we know that
the fiber~$\order{A}(o)$ is Morita equivalent to
the algebra~$\Lambda_r$ defined in \eqref{eq:fibre-over-closed-point}.
For each~$i\in Q_0=\{1,\ldots, r\}$, denote by~$S_i$ the simple~$\Lambda_r$-module
as defined in \eqref{eq:definition-simple-module},
and by~$P_i$ the unique indecomposable projective~$\Lambda_r$-module
such that~$P_i/{\rad P_i} \cong S_i$.

\begin{theorem}\label{thm:absorption-of-singularities}
    With the notation as above let~$i\in Q_0$.
    \begin{enumerate}[label = \roman*)]
        \item There is a semiorthogonal collection~$(S_{i+1},\ldots, S_{i-1})$ of~$\bbP^{\infty,2}$-objects in~$\bd(\Lambda_r)$.
        \item Let 
        \begin{equation}\label{eq:definition-categorical-absorption-category}
            \sfS_i = \langle S_{i+1},\ldots, S_{i-1} \rangle \subset \bd(\Lambda_r)
        \end{equation}
        be the triangulated subcategory generated by~$S_{i+1},\ldots, S_{i-1}$.
        Then there is a semiorthogonal decomposition
        \begin{equation}\label{eq:sod-ramified-fiber}
            \bd(\Lambda_r) = \langle \sfS_i, P_i\rangle.
        \end{equation}
        \item The triangulated subcategory~$\sfS_i$ absorbs singularities of~$\Lambda_r$.
    \end{enumerate}
\end{theorem}
Note that every~$S_i$ has infinite projective dimension and admits a 2-periodic projective resolution of the form
\begin{equation}\label{eq:projresolution}
    C^\bullet_i \colonequals(\ldots \xrightarrow{\mu_{[i+1,i]}} P_{i+1} \xrightarrow{\mu_{i,i+1}} P_i \xrightarrow{\mu_{[i+1,i]}} P_{i+1} \xrightarrow{\mu_{i,i+1}} P_i),
\end{equation}
where the maps in the cochain complex are given by left multiplication with the elements written above the arrows.
See \eqref{eq:cyclic-quiver} for their definition.

\begin{lemma}\label{lem:p-infty-2-objects}
    Let~$i\in Q_0$.
    \begin{enumerate}[label = \roman*)]
        \item The simple~$\Lambda_r$-module~$S_i$ is a~$\bbP^{\infty,2}$-object.
        \item The triangulated subcategory~$\langle S_i\rangle\subset \bd(\Lambda_r)$ generated by~$S_i$ is admissible.
    \end{enumerate}
\end{lemma}
\begin{proof}
    Let~$i\in Q_0$. By \Cref{rem:easier-p-infty-object},
    it suffices to show that~$\Ext_{\Lambda_r}^\bullet(S_i,S_i)\cong \field[\theta]$ with~$\deg \theta = 2$.
    Since~$\Hom_{\Lambda_r}(P_{i+1},S_i) = 0$, it follows from \eqref{eq:projresolution} that
    \begin{equation}
        \Ext_{\Lambda_r}^{k}(S_i,S_i) \cong \begin{cases}
            \field\theta^{k} & \text{if }k \in 2\bbZ_{\ge 0},\\
            0 & \text{else.}
        \end{cases}
    \end{equation}
    The map~$\theta\colon S_i \to S_i[2]$, can be explicitly described, using the projective resolution~$C^\bullet_i$,
    as the morphism of cochain complexes which is the identity map in each degree below~$-1$.

    For the admissibility of~$\langle S_i\rangle$,
    we use that the cochain complex~$M_i = (P_{i+1}\to P_i)\in \sfD^\perf(\Lambda_r)$
    of projective~$\Lambda_r$-modules, which is
    concentrated in degrees~$\{-1,0\}$,
    represents the third object in the canonical self-extension \eqref{eq:canonical-self-extension} of~$S_i$.
    The algebra~$\Lambda_r$ is Gorenstein in the sense of \cite[Assumption 0.1]{MR4133519},
    because every injective~$\Lambda_r$-module is projective as well.
    Hence by \cite[Proposition 6.9]{MR4868123} the cochain complex~$M_i$ is homologically left and right finite-dimensional.
    It follows from \cite[Lemma 2.10]{MR4698898} that the triangulated category~$\langle S_i\rangle$ generated by~$S_i$ is admissible.
\end{proof}

\begin{lemma}\label{lem:orthohonality-of-simple-modules}
    For all~$1\le i,j \le r$, one has
    \begin{equation}
        \Hom_{\bd(\Lambda_r)}(S_j,S_i[\ell]) = 0 \quad \text{for all }\ell\in \bbZ,
    \end{equation}
    if and only if~$j \notin \{i,i+1\}$.
\end{lemma}
\begin{proof}
    From \eqref{eq:projresolution} and~$\Hom_{\Lambda_r}(P_j,S_k) = 0$ unless~$j = k$,
    one obtains that~$\Ext_{\Lambda_r}^\bullet(S_j, S_k) = 0$ unless~$k \in \{j,j+1\}$.
    Therefore, the collection is semiorthogonal.
\end{proof}

\begin{lemma}\label{lem:sod-ramified-fiber}
    With the notation from \Cref{thm:absorption-of-singularities},
    let~$P_i$ be the indecomposable projective
    and~$I_i$ the indecomposable injective~$\Lambda_r$-module
    associated with the vertex~$i\in Q_0$. Then
    \begin{align}
        \label{eq:sod-indec-proj}\bd(\Lambda_r) = \langle \sfS_{i}, P_i\rangle,\\
        \label{eq:sod-indec-inj}\bd(\Lambda_r) = \langle I_i,\sfS_{i}\rangle.
    \end{align}
    are semiorthogonal decompositions of~$\bd(\Lambda_r)$.
\end{lemma}
\begin{proof}
    By rotational symmetry, it suffices to consider~$i = r$.
    Both collections are semiorthogonal by \Cref{lem:orthohonality-of-simple-modules}, and because 
    \begin{equation}
        \Hom_{\Lambda_r}(P_r,S_j) = \Hom_{\Lambda_r}(S_j,I_r)=0 \quad\text{for all }1\le j < r.
    \end{equation}
	It remains to show that~$S_1,\ldots, S_{r-1}, P_r$ generate~$\bd(\Lambda_r)$.
    Since~$\bd(\Lambda_r)$ is generated by~$\Lambda_r$, it suffices to show
    that all indecomposable projective~$\Lambda_r$-modules~$P_1,\ldots, P_r$ are in
    the smallest triangulated subcategory~$\sfT$ generated by~$S_1,\ldots, S_{r-1}, P_r$.
    
    From the canonical self-extension for~$S_i$, we obtain that~$M_1 = (P_2 \to P_1), \ldots, M_{r-1} = (P_{r}\to P_{r-1})$
    belong to~$\sfT$. For every~$i\in 1,\ldots, r-1$, one obtains a distinguished triangle
    \begin{equation}
        P_{i} \to M_i \to P_{i+1}\to P_{i}[1].        
    \end{equation}
    Since~$\sfT$ is closed under taking cones and~$P_r \in \sfT$, 
    the claim follows.
\end{proof}
We are now ready to prove the main assertion of this section.
\begin{proof}[Proof of \Cref{thm:absorption-of-singularities}]
    Part (i) is a consequence of \Cref{lem:p-infty-2-objects} and \Cref{lem:orthohonality-of-simple-modules}. Statement (ii) is part of \Cref{lem:sod-ramified-fiber}. 
    
    \Cref{lem:sod-ramified-fiber} implies in particular that~$\sfS_i$ is admissible,
    and both~${}^\perp\sfS_i$ as well as~$\sfS_i^\perp$ are generated by an exceptional object, hence smooth and proper. Therefore,~$\sfS_i$ absorbs singularities.
\end{proof}

\subsection{Deformation absorption of singularities}\label{sec:deformation-absorption-of-singularities}
Throughout this section
let~$(C,\order{A})$ be a pair of a curve~$C$,
and a hereditary~$\calO_C$-order~$\order{A}$
with ramification locus~$\Delta_\order{A} = \{o\}$ a single closed point
of ramification index~$r\ge 1$.
Since orders are generically Azumaya,
this can always be achieved for every hereditary~$\calO_C$-order~$\order{A}$
by shrinking~$C$ to a Zariski open neighborhood around a chosen ramification point.

Given a closed point~$p \in C$, we denote by~$i_p\colon \Spec\field(p)\to C$ the inclusion.
Building on the perspective of \Cref{rem:base-change-structure-morphism},
we consider the base change diagram
\begin{equation}\label{eq:fibre-diagram}
	\begin{tikzcd}
		(\Spec\field(p),\order{A}(p))\arrow[rr,hook,"\mfi_p ={(i_p, \text{id}_{\order{A}(p)})}"]\arrow[d, "\mff_p"] && (C,\order{A})\arrow[d,"\mff"]\\
		\Spec\field(p)\arrow[rr,hook,"i_p"] && C
	\end{tikzcd}
\end{equation}
of the structure morphism~$\mff \colon (C,\order{A})\to C$ along~$i_p$.
If~$p\neq o$, the~$\field(p)$-algebra~$\order{A}(p)$ is a matrix algebra
and therefore Morita equivalent to~$\field(p)$.
From \Cref{lem:morita-equivalence-to-cyclic-algebra}, one knows that~$\order{A}(o)$ is Morita equivalent
to the algebra~$\Lambda_r$ from \eqref{eq:fibre-over-closed-point}.

Since~$\mff$ is an extension and every $\calO_C$-order is flat,
the morphism~$i_p$ is faithful for~$\mff$ by \Cref{lem:faithful-base-change}.
Moreover, using \Cref{lem:D-perf-globally,lem:finite-cohomological-amplitude}
for hereditary $\calO_C$-orders,
we may appeal to \Cref{thm:main-base-change-theorem}
for the base change of
strong semiorthogonal decompositions of~$\bd(C,\order{A})$
along $\mfi_p$.
Our main result of this section is the following.

\begin{theorem}\label{thm:deformation-absorption}
    Let~$\order{A}$ be a hereditary~$\calO_C$-order
    ramified over~$\Delta_\order{A} = \{o\}$
    with ramification index~$r\in \bbZ_{\ge 1}$.
    For each~$i \in\{1,\ldots, r\}$ there is a strong~$C$-linear semiorthogonal decomposition
    \begin{equation}\label{eq:main-sod}
        \bd(C,\order{A}) =
        \langle \mfi_{o,*}S_{i+1},\ldots,  \mfi_{o,*}S_{i-1}, \sfD\rangle
    \end{equation}
    such that
    \begin{enumerate}[label = \roman*)]
        \item the sequence~$\mfi_{o,*}S_{i+1},\ldots,  \mfi_{o,*}S_{i-1}$ is exceptional,
        \item the admissible subcategory~$\sfD$ is smooth and proper over~$\bd(C)$,
        \item the fibers of~$\sfD$ over~$p\in C$ are equivalent to~$\bd(\modules\field(p))$.
    \end{enumerate}
\end{theorem}
In light of \cite[Theorem 1.5]{MR4698898}
we say that~$\sfS_i= \langle S_{i+1},\ldots, S_{i-1}\rangle$ (from \Cref{thm:absorption-of-singularities})
\emph{provides a deformation absorption of singularities of~$\Lambda_r$ with respect to}~$(C,\order{A})$,
i.e.~it is admissible in~$\bd(C,\order{A})$.

We split the proof of \Cref{thm:deformation-absorption} into three steps.
In the first step, we show that,
as in the commutative case of \cite[Theorem 1.8]{MR4698898},
a semiorthogonal collection of~$\bbP^{\infty,2}$-objects on the singular fiber pushes forward to an
exceptional collection along the noncommutative smoothing.
The idea of the proof is similar, but requires the extension to~$\ncSch_\field$. Let us spell out the details.
\begin{lemma}\label{lem:pushforward-of-p-infty-2-objects}
    The semiorthogonal collection of~$\bbP^{\infty,2}$-objects
   ~$(S_{i+1},\ldots, S_{i-1})$ of~$\bd(\Lambda_r)$
    pushes forward to an exceptional collection
    \begin{equation}\label{eq:exceptional-sequence-of-simples}
	    (\mfi_{o,*}S_{i+1},\ldots, \mfi_{o,*}S_{i-1})\quad\text{in } \bd(C,\order{A}).
    \end{equation}
\end{lemma} 

\begin{proof}
    By the rotational symmetry of the question,
    it suffices to consider the case~$i = r$.

	Since~$\mfi$ is the base change of~$(C,\order{A})\to C$ along~$i_o\colon \Spec\field(o)\to C$,
    the object~$\bfL\mfi^*\mfi_{o,*} S_k$ fits into the distinguished triangle
	\begin{equation}\label{eq:distinguished-triangle-divisor}
		S_k[1] \to \bfL\mfi_o^*\mfi_{o,*} S_k \to S_k \xrightarrow{\theta}  S_k[2]
	\end{equation}
    for Cartier divisors with trivial normal bundle \cite[Section 4.2]{MR4698898}.
	The map~$S_k\to S_{k}[2]$ is non-zero,
    because the pullback~$\bfL\mfi^*$ preserves perfect complexes
    and~$S_k\notin \sfD^{\perf}(\Lambda_r)$. It follows from the comparison of this triangle
    to the canonical self-extension \eqref{eq:canonical-self-extension} that
	\begin{equation}
		\Ext^{\bullet}_{\order{A}}(\mfi_{o,*} S_k,\mfi_{o,*} S_k)
        \cong \Ext_{\Lambda_r}^\bullet(\bfL\mfi^*\mfi_{o,*} S_k,S_k)
        \cong \Ext_{\Lambda_r}^\bullet(M_k,S_k)\cong \field[0].
	\end{equation}
	Therefore, each object is exceptional.

	Similarly, the distinguished triangle \eqref{eq:distinguished-triangle-divisor} leads to
    the vanishing of~$\Ext^{\bullet}_{(C,\order{A})}(\mfi_{o,*} S_k,\mfi_{o,*} S_j)$ for~$1\le j<k\le r-1$,
    because~$M_k = (P_{k+1}\to P_{k})$ maps only trivially to~$S_j$.
\end{proof}
Together with \Cref{thm:absorption-of-singularities}
this lemma already shows that~$\sfS_i\subset \bd(\Lambda_r)$ provides a deformation absorption of singularities
with respect to~$(C,\order{A})$.
In the next step we prove that \eqref{eq:main-sod} is a~$C$-linear semiorthogonal decomposition.
\begin{lemma}\label{lem:linear-sod}
	There is a~$C$-linear semiorthogonal decomposition 
	\begin{equation}\label{eq:desired-sod}
		\bd(C,\order{A}) = \langle \mfi_{o,*} S_{i+1},\ldots, \mfi_{o,*} S_{i-1}, \sfD\rangle.
	\end{equation}
\end{lemma}

\begin{proof}
	A sequence of exceptional objects generates
    an admissible subcategory by \cite[Theorem 3.2]{MR992977}.
    Hence we obtain a semiorthogonal decomposition of~$\bd(C,\order{A})$
    with~$\sfD = {}^\perp\langle \mfi_{o,*} S_{i+1},\ldots, \mfi_{o,*} S_{i-1}\rangle$.

    As before we consider only the case~$i = r$.
	For the~$C$-linearity, we use that each of the admissible subcategories~$\langle\mfi_{o,*} S_k\rangle$ is~$C$-linear.
    More precisely, for each~$S_k\in \bd(\Lambda_r)$
    and every~$\calF\in \sfD^\perf(C)$,
    the projection formula implies 
	\begin{equation}
		(\mfi_{o,*} S_k) \otimes_{C}^{\bfL} f^* \calF
        \cong \mfi_{o,*}\left(S_k\otimes_{\field} \bfL( f_o\circ i_o)^* \calF\right).
	\end{equation}
	
    Since~$\bfL( f_o\circ i_o)^* \calF$ can be represented by 
    a bounded cochain complex of~$\field$-vector spaces,
    one obtains~$S_k\otimes_{\field} \bfL( f_o\circ i_o)^* \calF \in \langle S_k\rangle$.
    Since~$\mfi_{o,*}$ is exact and commutes with direct sums,
    the~$C$-linearity of~$\langle \mfi_{o,*} S_k\rangle$ follows.
    The~$C$-linearity of the complement~$\sfD$ follows from \Cref{lem:linearity-of-complement}.
\end{proof}

For the last lemma recall the notion of locally projective~$\order{A}$-modules
purely of one type from \Cref{def:type-of-maximal-overorder}
and the characterization of maximal overorders of~$\order{A}$
from \Cref{prop:classification-of-overorders}.
For~$i \in Q_0 = \{1,\ldots, r\}$ denote
by~$\order{B}_i$ the (unique) maximal overorder of~$\order{A}$
which is purely of type~$i$ at~$o$.
\begin{lemma}\label{lem:embedding-the-curve}
    The component~$\sfD$ in the semiorthogonal decomposition
    from \Cref{lem:linear-sod} is equivalent to~$\bd(C)$ given by 
    the thick closure of the embedding 
    \begin{equation}
        \mfj_{\order{B}_i,*}\colon \bd(C,\order{B}_i) \to \bd(C,\order{A}),
    \end{equation}
    where~$\order{B}_i$ is purely of type~$i$ at~$o$.
    In particular,~$\sfD$ is an admissible subcategory.
\end{lemma}

\begin{proof}
    Each of the maximal orders~$\order{B}_i$ is in fact Azumaya.
    We consider the case~$i =r$.
    Since~$\field$ is algebraically closed, we have~$\coh(C,\order{B}_r)\simeq \coh(C)$.

    By \Cref{lem:properties-of-morphisms-by-maximal-overorders}, the pushforward and the pullback are both exact,
    and the pushforward is fully faithful.
    This implies that
    \begin{equation}
        \mfj_{\order{B}_{r},*}\colon \bd(C,\order{B}_r)\to \bd(C,\order{A})
    \end{equation}
    is fully faithful on the bounded derived category as well.

    We need to show that~$\mfj_{\order{B}_{r},*}\bd(C,\order{B}_r) = \sfD$.
    Using \Cref{lem:rel-criterion-semiorthogonality}, one has to show that~$M \in \mfj_{\order{B}_r,*}\bd(C,\order{B}_r)$
    if and only if~$\bfR\inthom_{\order{A}}(M,\mfi_{o,*}S_k) = 0$ for every~$k = 1,\ldots,r-1$.
    Note that~$\mff$ is an extension.
    Therefore~$\bfR f_*$ is the identity.  

    Since~$M\in \bd(C,\order{A})$,
    we can calculate~$\bfR\inthom_{\order{A}}(M,\mfi_{o,*}S_k)$ 
    by replacing~$M$ by a bounded cochain complex of projective~$\order{A}$-modules.
    Moreover~$\mfi_{o,*}S_k \in \coh(C,\order{A})$ and it is supported at~$o\in C$.
    Thus, it suffices to show 
    \begin{equation}
        \bfR\Hom_{\order{A}_o}(M_o,\mfi_{o,*}S_k) = 0.
    \end{equation}

    Assume that~$M \in \mfj_{\order{B}_r,*}\sfD^b(C,\order{B}_r)$.
    As~$\order{B}_r$ is purely of type~$r$ at~$o$, it follows from \Cref{prop:classification-of-overorders} and its proof
    that~$\order{B}_r \cong \End_{\calO_{C,o}}(L_o^{(r)})$, where~$L_o^{(r)}$ is
    the indecomposable projective~$\order{A}_{C,o}$-module
    defined in \eqref{eq:locally-indec-projectives}.
    Hence~$M_o$ can be expressed as an iterated cone of direct sums of~$L_o^{(r)}$.
    Since~$k\neq r$, it follows that~$M\in \sfD$ as 
    \begin{equation}
        \bfR\Hom_{\order{A}_o}(L_o^{(r)},\mfi_{o,*}S_k) = \Hom_{\order{A}_o}(L_o^{(r)},\mfi_{o,*}S_k) = 0.
    \end{equation}
    Vice versa, assume that~$M\in \sfD$.
    Restricting to~$(\Spec(\calO_{C,o}),\order{A}_{o})$,
    we can represent~$M_o$ by a bounded cochain complex~$Q^\bullet$ of projective~$\order{A}_o$-modules.
    Since we have by assumption that~$\Hom_{\bd(\order{A}_o)}(Q^\bullet,\mfi_{o,*}S_k) = 0$ for every~$i = 1,\ldots, r-1$,
    the cochain complex~$Q^\bullet$ belongs to the subcategory generated by~$L_o^{(r)}$.
    It follows that~$M$ must belong to the~$C$-linear subcategory of~$\bd(C,\order{A})$
    generated by~$\order{B}_r$, i.e.~$M\in \mfj_{r,*}\bd(C,\order{B}_r)$.

    The right admissibility of~$\sfD$ is automatic
    and left admissibility follows from the adjunction~$\mfj_{\order{B}_r}^* \dashv \mfj_{\order{B}_r,*}$.
\end{proof}

\begin{proof}[Proof of \Cref{thm:deformation-absorption}]
    The fact that the semiorthogonal decomposition \eqref{eq:main-sod} is strong follows from \Cref{lem:embedding-the-curve}.
    We have shown in \Cref{lem:pushforward-of-p-infty-2-objects} that~$(\mfi_{o,*}S_{i+1},\ldots, \mfi_{o,*}S_{i-1})$
    is an exceptional collection in~$\bd(C,\order{A})$.
    From the equivalence~$\sfD \simeq \bd(C)$ of \Cref{lem:embedding-the-curve},
    it follows that~$\sfD$ is smooth and proper over~$\bd(C)$.
    For the third point,
    note that over each point~$p\neq o$,
    the restriction~$\order{A}(p)$ is Azumaya,
    and hence Morita equivalent to a point. 
    Moreover,~$\sfD_o$ is the admissible subcategory
    generated by the~$r$-th indecomposable projective~$\Lambda_r$-module, which is exceptional.
\end{proof}

\begin{remark}
    There is a two-dimensional analogue given by \emph{tame} orders of global dimension 2.
    By \cite[Theorem 1.1]{MR978602} a tame order on a surface is uniquely determined by its overorders.
    Moreover Theorem 1.14 of \emph{op.~cit.} hints to a similar decomposition of the derived category of a tame order
    as in \Cref{lem:embedding-the-curve} using a maximal overorder.
    The precise shape of such a decomposition
    must take into account that maximal orders on surfaces
    are not necessarily Azumaya.
    In light of the recently developed stacks--orders dictionary \cite{2206.13359v2} in dimension two,
    such a decomposition would be interesting.
\end{remark}

\paragraph{Periodicity of the semiorthogonal decomposition.}
As an application we show that the semiorthogonal decomposition \eqref{eq:main-sod} of~$\bd(C,\order{A})$
is~$2r$-periodic.
Let 
\begin{equation}
    \sfT = \langle \sfA, \sfB\rangle
\end{equation}
be a semiorthogonal decomposition. 
We denote the \emph{right dual semiorthogonal decomposition}
by~$\sfT = \langle \sfB, \bbR_{\sfB}\sfA\rangle$, where~$\bbR_{\sfB}$ is the right mutation functor
as defined in \cite[\S 2]{MR992977}.
By \cite[Definition 4.2]{MR4790551} a semiorthogonal decomposition~$\sfT = \langle\sfA,\sfB \rangle$ is 
\emph{$N$-periodic} if the~$N$th right dual is again the original decomposition.

In \cite[Section 4]{MR4790551} the periodicity of a semiorthogonal decomposition
for the derived category of a root stack is studied.
We provide the same result for the semiorthogonal decomposition \eqref{eq:main-sod}
thereby explaining the connection between the~$r$ different versions of \eqref{eq:main-sod}.

\begin{theorem}
    The semiorthogonal decomposition \eqref{eq:main-sod} is~$2r$-periodic.
\end{theorem}

\begin{proof}
    Start with the semiorthogonal decomposition
    \begin{equation}
        \bd(C,\order{A}) = \langle \mfi_{o,*}S_{i+1},\ldots, \mfi_{o,*}S_{i-1}, \mfj_{\order{B}_{i,*}}\bd(C,\order{A})\rangle.
    \end{equation}
    By the proof of \Cref{lem:embedding-the-curve}, the category~$\mfj_{\order{B}_{i,*}}\bd(C,\order{A})$
    is generated by~$\{P_i \otimes_C \calL_{\alpha}\}$,
    where~$\{\calL_{\alpha}\}_{\alpha}$ is a generating set of~$\bd(C)$,
    and~$P_i$ is the locally projective~$\order{A}$-module purely of type~$i$ at~$o$
    such that~$P_{i,o} \cong L_o^{(i)}$ is indecomposable.
    See \eqref{eq:locally-indec-projectives} for a definition.

    By \cite[Theorem 1]{MR738217}~$\bd(C,\order{A})$ possesses a Serre functor~$\bbS_\order{A}\colon \bd(C,\order{A})\to \bd(C,\order{A})$
    given by~$\bbS_\order{A}(M)= M\otimes_{\order{A}}\omega_{\order{A}}[1]$
    with dualizing bimodule~$\omega_\order{A} = \inthom_X(\order{A},\omega_X)$.
    The Serre functor satisfies~$\bbS_{\order{A}}(P_i) = P_{i+1}[1]$,
    and~$\bbS_{\order{A}}(\mfi_{o,*}S_i) = \mfi_{o,*}S_{i+1}$.

    If we denote by~$\sfA = \langle \mfi_{o,*}S_{i+1},\ldots, \mfi_{o,*}S_{i-1}\rangle$,
    and by~$\sfB = \mfj_{\order{B}_{i,*}}\bd(C,\order{A})$, it follows from \cite[Proposition 3.6]{MR1039961}
    that~$\bbR_{\sfB}(\sfA) = \bbS(\sfA)$.
    Hence, after~$2r$ times taking the right dual of~$\sfB$,
    the category~$\sfB$ is replaced by~$\bbS^{r}(\sfB)$.
    Since~$\bbS^r(\sfB)$ is generated by~$\{\bbS^r(P_i)\otimes_C\calL_{\alpha}\}_\alpha$,
    and~$\sfB$ is triangulated, it follows that~$\bbS^r(\sfB) = \sfB$.
    Similarly, we have that~$\bbS^{r}(\sfA) = \sfA$.
\end{proof}

\subsection{The dictionary between hereditary orders and smooth root stacks}\label{sec:dictionary-orders-stacks}
There is a stacks--orders dictionary in dimension one and two \cite{MR2018958,2206.13359v2,2410.07620}.
Let~$C$ be a quasi-projective curve.
By \cite[Corollary 7.8]{MR2018958}, resp.~\cite[\S 2.2]{2410.07620} in the language of root stacks,
the dictionary relates the following two objects.
\begin{enumerate}[label = \roman*)]
    \item Let~$\order{A}$ be a hereditary~$\calO_C$-order with ramification divisor~$\Delta = \{p_1,\ldots,p_m\}$
    and ramification indices~$\bfr =(r_1,\ldots, r_m)\in \bbN^{m}$, and
    \item denote by~$\pi \colon \sqrt[\bfr]{C;\Delta}\to C$ the iterated root stack over~$C$ 
    by doing the~$r_i$-th root construction at~$p_i\in \Delta_{\order{A}}$.
\end{enumerate}
\begin{theorem}[Chan--Ingalls]
    With the notation as above, there is an equivalence of categories
    \begin{equation}
        \coh(C,\order{A}) \simeq \coh(\sqrt[\bfr]{C;\Delta_\order{A}}).
    \end{equation}
\end{theorem}
Therefore, there is a stacky version of \Cref{thm:deformation-absorption}.
Assume that~$\Delta = \{o\}$ and do the~$r$-th root construction at~$o$.
Then the singular fiber is described by
\begin{equation}
    \Spec(\field(o))\times_C \sqrt[r]{C;\Delta} \cong \left[\Spec\left(\frac{\field[t]}{(t^r)}\right)/\mu_r\right],
\end{equation}
where~$\mu_r$ is the group scheme of~$r$-th roots of unity, and the~$r$-th primitive root acts by multiplication on~$t$.
Coherent sheaves on the singular fiber are given by 
\begin{equation}
    \coh([\Spec(\field[t]/(t^r))/\mu_r])
    \simeq \coh^{\mu_r}(\field[t]/(t^r))
    \simeq \coh (\field[t]/(t^r)*\mu_r),
\end{equation}
where~$\field[t]/(t^r)*\mu_r$ is the skew group algebra.

On the other hand each closed point in~$\sqrt[r]{C;\Delta}$
(i.e.~a morphism~$\Spec \field \to \sqrt[r]{C;\Delta}$)
factors through its residual gerbe~$\text{B}\text{Aut}(p)$,
where~$p\in C$ is obtained by postcomposing with the map to the coarse moduli space~$C$.
The residual gerbe is only nontrivial at~$o\in C$,
where we have~$\text{B}\mu_r = [(\Spec\field(o))/\mu_r]$,
and
\begin{equation}
    \coh(\text{B}\mu_r) \simeq \modules(\field[\mu_r]).
\end{equation}
The group algebra~$\field[\mu_r]$ is semisimple with~$r$ simple modules.
Following the terminology of \cite[Section 2.2]{MR4280492},
we obtain therefore~$r$ generalized points~$(o,\zeta_1),\ldots, (o,\zeta_r)$,
each corresponding to one character of~$\field[\mu_r]$.
Denote by~$\calO_{o,\zeta_i}$ the irreducible~$\mu_r$-representation corresponding to~$(o,\zeta_i)$.

By \cite[Example 2.4.3]{MR2306040}, the residual gerbe over~$o$ embeds as a closed substack
into the fiber~$\Spec \field(o) \times_C \sqrt[r]{C;\Delta}$ so that we obtain a commutative diagram
\begin{equation}
\begin{tikzcd}
    \left[\Spec\field(o)/\mu_r\right]\arrow[dr]\arrow[r, "j"]&
    \left[\Spec(\field[t]/(t^r))/\mu_r\right]\arrow[r, "\iota_o"]\arrow[d] &
    \sqrt[r]{C;\Delta}\arrow[d, "\pi"]\\
    & \Spec\field(o)\arrow[r, "i_o"]&C
\end{tikzcd}
\end{equation}
On the level of algebras,~$j$ corresponds
to the~$\field$-algebra homomorphism~$\field[t]/(t^r)*\mu_r \to \field[\mu_r]$,
which identifies the irreducible module~$\calO_{o,\zeta_i}$
with a simple~$\field[t]/(t^r)*\mu_r$-module, denoted by~$\widetilde{\calO}_{o,\zeta_i}$.
We can use the modules~$\widetilde{\calO}_{o,\zeta_i}$ for the version
of \Cref{thm:absorption-of-singularities}
for stacky curves.
\begin{corollary}\label{cor:absorption-singularities-stacks}
    Let~$i \in \{1,\ldots, r\}$.
    The collection~$\widetilde{\calO}_{o,\zeta_{i+1}}, \ldots, \widetilde{\calO}_{o,\zeta_{i-1}}$
    (counted modulo~$r$) is a semiorthogonal collection
    of~$\bbP^{\infty,2}$-objects in~$\bd(\Spec(\field(o))\times_C \sqrt[r]{C;\Delta})$.

    Moreover, the smallest triangulated category~$\sfO_i\subset \bd(\Spec(\field(o))\times_C \sqrt[r]{C;\Delta})$
    containing~$\widetilde{\calO}_{o,\zeta_{i+1}}, \ldots, \widetilde{\calO}_{o,\zeta_{i-1}}$
    absorbs singularities of~$\Spec(\field(o))\times_C \sqrt[r]{C;\Delta}$.
\end{corollary}
\begin{proof}
    This follows from \Cref{thm:absorption-of-singularities} and
    the~$\field$-algebra isomorphism~$\field[t]/(t^r)*\mu_r \cong \Lambda_r$,
    which identifies the simple~$\field[t]/(t^r)*\mu_r$-module~$\widetilde{\calO}_{o,\zeta_i}$
    with the simple~$\Lambda_r$-module~$S_i$.
\end{proof}

Using \cite[Theorem 4.7]{MR3573964} we obtain the deformation absorption result
for~$\sqrt[r]{C;\Delta}$ as well.
\begin{theorem}\label{thm:deformation-absorption-stacks}
    Let~$\sqrt[r]{C;o}\to C$ be smooth stacky curve
    with nontrivial stabilizer~$\mu_r$ over the closed point~$o\in C$.
    For each~$i \in\{1,\ldots, r\}$ there is a strong~$C$-linear semiorthogonal decomposition
    \begin{equation}\label{eq:main-sod-stacks}
        \bd(\sqrt[r]{C;o}) =
        \langle \iota_{o,*}\widetilde{\calO}_{o,\zeta_{i+1}},\ldots,  \iota_{o,*}\widetilde{\calO}_{o,\zeta_{i-1}}, \sfD\rangle
    \end{equation}
    such that
    \begin{enumerate}[label = \roman*)]
        \item the sequence~$\iota_{o,*}\widetilde{\calO}_{o,\zeta_{i+1}},\ldots,  \iota_{o,*}\widetilde{\calO}_{o,\zeta_{i-1}}$ is exceptional,
        \item the admissible subcategory~$\sfD$ is smooth and proper over~$\bd(C)$,
        \item the fibers of~$\sfD$ over~$p\in C$ are equivalent to~$\bd(\modules\field(p))$.
    \end{enumerate}
    In other words~$\sfO_{i}$ provides a deformation absorption of singularities
    of~$\left[\Spec(\field[t]/(t^r))/\mu_r\right]$ with respect to the smoothing~$\sqrt[r]{C;o}\to C$.
\end{theorem}
\begin{proof}
    We have that~$\iota_{o,*}\widetilde{\calO}_{o,\zeta_k} = \bfR(\iota_{o,*}\circ j)\calO_{o,\zeta_k}$.
    Therefore, \cite[Theorem 4.7]{MR3573964} provides
    the semiorthogonal decomposition \eqref{eq:main-sod-stacks} for~$i=0$,
    where~$\sfD = \bfL \pi^*\bd(C)$.
    For~$i>0$ the semiorthogonal decomposition follows by \cite[Theorem 4.3]{MR4790551}.
\end{proof}

\appendix
\section{A noncommutative base change formula}
\label{appendix:noncommutative-base-change-formula}
The goal of the appendix
is to give a version of
Kuznetsov's base change formula \cite[Theorem 5.6]{MR2801403}
for the bounded derived category $\bd(X,\order{A})$
of a coherent ringed scheme $(X,\order{A})$.
After restricting our attention to coherent $\calO_X$-algebras $\order{A}$
of finite global dimension,
we only have to make small modifications of the proofs of \emph{op.~cit.~}to arrive at \Cref{thm:main-base-change-theorem},
the base change formula for noncommutative schemes.
We denote
\begin{itemize}
    \item by~$\sfD(X,\order{A}) = \sfD(\QCoh(X,\order{A}))$ the \emph{unbounded derived category of quasicoherent~$\order{A}$-modules},
    and by~$\sfD^*(X,\order{A})$, for~$* \in \{+, -, \sfb\}$, its bounded below, bounded above, resp.~bounded
    derived category;
    \item by~$\sfD_\coh(X,\order{A}) = \sfD_\coh(\QCoh(X,\order{A}))$
    the \emph{derived category of quasicoherent~$\order{A}$-modules with coherent cohomology}, and
    by~$\bd(X,\order{A}) \colonequals \sfD^b_\coh(\QCoh(X,\order{A}))$ the \emph{bounded derived category} of~$(X,\order{A})$;
    \item by~$\sfD^\perf(X,\order{A})$ the \emph{category of perfect complexes}
    consisting of objects
    which are represented by complexes that are locally quasi-isomorphic to bounded complexes of locally projective~$\order{A}$-modules;
    \item by~$\sfD^{[a,b]}(X,\order{A})\subset \sfD(X,\order{A})$,
    for~$a\le b \in \bbZ$,
    all~$M\in \sfD(X,\order{A})$ such that the cohomology sheaf~$\calH^{i}(M)$ vanishes for~$i\notin [a,b]$.
    If~$a = -\infty$, we write~$\sfD^{\le b}(X,\order{A})$, 
    and if~$b = \infty$, we write~$\sfD^{\ge a}(X,\order{A})$.
\end{itemize}
\subsection*{Perfect complexes}
Recall that an object~$P \in \sfD(X,\order{A})$ is \emph{compact}
if~$\Hom_{\sfD(X,\order{A})}(P,-)$ commutes with filtered colimits.
\begin{remark}\label{rem:perfect-complexes-and-compact-objects}
    For~$(X,\order{A})$ a coherent ringed scheme, \cite[Proposition 3.14]{MR3695056} shows
    that~$\sfD(X,\order{A})$ is compactly generated,
    and the compact objects are~$\sfD(X,\order{A})^\compact = \sfD^\perf(X,\order{A})$.
\end{remark}

Using the identification \eqref{eq:coherent-cohomology-and-bounded-derived-category},
it is straightforward
to extend \cite[Proposition 2.3.1]{MR1106918} to coherent ringed schemes~$(X,\order{A})$, where~$\order{A}$ has finite global dimension.
For $\order{A} = \calO_X$, it says that a perfect complex on a quasi-projective scheme~$X$
is quasi-isomorphic to a bounded cochain complex of locally free sheaves.
\begin{lemma}\label{lem:D-perf-globally}
    Let~$(X,\order{A})$ be a coherent ringed scheme.
    \begin{enumerate}[label = \roman*)] 
        \item If~$\order{A}$ has finite global dimension or~$X = \Spec \field$ is a point,
        each perfect complex is globally quasi-isomorphic to a bounded cochain complex of locally projective modules.
        \item If~$\order{A}$ has finite global dimension, then~$\sfD^\perf(X,\order{A}) = \bd(X,\order{A})$.
    \end{enumerate}
\end{lemma}
\begin{proof}
    Let~$M\in \sfD^\perf(X,\order{A})$.
    Since~$X$ is quasi-compact and the cohomology sheaf~$\calH^{\ell}(M)$ does only depend
    on the~$\calO_X$-module structure,
    it follows that~$\calH^{\ell}(M)$ is coherent for all~$\ell \in \bbZ$,
    and non-zero for only finitely many~$\ell$.
    From \eqref{eq:coherent-cohomology-and-bounded-derived-category} one sees
    that~$M \in \bd(X,\order{A}) \simeq \bd(\coh(X,\order{A}))$ since~$\order{A}$ is coherent.
    Hence every perfect complex is quasi-isomorphic
    to a bounded cochain complex of coherent $\order{A}$-modules.

    If~$X = \Spec\field$ the statements follow from the fact that every module $\order{A}$-module admits a projective resolution and by definition of perfect complexes.
    Hence, we pass to $\order{A}$ of finite global dimension.
    Let~$M \in \bd(X,\order{A})$.
    By \cite[Proposition 3.7]{MR3695056}, there is a locally projective resolution~$P^\bullet \xrightarrow{\sim} M$, 
    with~$P^\bullet \in \sfD^{-}(\coh(X,\order{A}))$ bounded above.
    Since~$M$ is bounded on both sides,
    there exists~$n\ll 0$ such that~$\calH^k(P^\bullet) = 0$ for all~$k\le n$.
    Then the canonical truncation~$N^\bullet = \tau_{\ge n}P^\bullet$ (cf.~\cite[{\href{https://stacks.math.columbia.edu/tag/0118}{Section 0118}}]{stacks-project})
    is quasi-isomorphic to~$P^\bullet$.
    The canonical truncation consists of locally projective~$\order{A}$-modules
    except at the~$n$-th position, we have~$N^n = \Coker(d^{n-1}\colon P^{n-1}\to P^n) \cong \Im(d^n)$.
    Since~$N^{n}$ is a coherent~$\order{A}$-module,
    it admits a locally projective resolution~$\varepsilon\colon Q^\bullet \to N^{n}$,
    which is of finite length if~$\order{A}$ is of finite global dimension.
    Replacing~$N^{n}$ by its locally projective resolution,
    we obtain a bounded cochain complex
    \begin{equation}
        L^\bullet = \left(\ldots \to Q^{-1}\to Q^{0}\xrightarrow{\iota\circ\varepsilon} P^{n+1}\to \ldots\right)
    \end{equation}
    of locally projective~$\order{A}$-modules
    and an induced morphism of cochain complexes $L^\bullet \to N^\bullet$.
    Here~$\iota\colon N^n\to P^{n+1}$ is the inclusion of the image of $d^n$.
    This is a quasi-isomorphism.
    All in all, we obtain that~$M$ quasi-isomorphic to~$L^\bullet \in \sfD^\perf(X,\order{A})$.
\end{proof}

Recall from \cite[Definition 1.6]{MR2437083} that an object~$M\in \bd(X,\order{A})$ is \emph{homologically finite}
if for every~$N\in \bd(X,\order{A})$
the~$\field$-vector space~$\bigoplus_{i\in \bbZ}\Ext_{\order{A}}^{i}(M,N)$ is finite-dimensional.
The following generalizes \cite[Proposition 1.11]{MR2437083} to the noncommutative setting.
\begin{lemma}\label{lem:perfect-is-homologically-finite}
    Let~$(X,\order{A})$ be a coherent ringed scheme
    such that~$\order{A}$ is of finite global dimension
    or~$X = \Spec\field$ is a point.
    An object~$M\in \bd(X,\order{A})$ is perfect if and only if it is homologically finite.
\end{lemma}
\begin{proof}
    We explain how to modify the proof in \cite[Proposition 1.11]{MR2437083} so that it works in our setting.
    Since every homologically finite object is bounded,
    by \cite[Proposition 2.9]{MR2403307} it is perfect.
    Vice versa, assume that~$M$ is perfect.
    By \Cref{lem:D-perf-globally} it is represented by a bounded cochain complex~$P^\bullet$ of locally projective~$\order{A}$-modules.
    The spectral sequence
    \begin{equation}
        \tH^p(X,\calH^q(\inthom_{\order{A}}(P^\bullet, N))) \quad \Rightarrow \quad \Ext_{\order{A}}^{p+q}(M,N)
    \end{equation}
    is concentrated in~$p\in [0,\dim X]$.
    It is zero in the~$q$-direction outside the bounds of~$P^\bullet$.
    Moreover, each~$\field$-vector spaces on the~$E_2$-page is finite-dimensional,
    because~$\order{A}$ is coherent.
    Therefore,~$\Ext_{\order{A}}^{i}(M,N)$ is finite-dimensional for all~$i\in \bbZ$
    and non-zero for only finitely many~$i \in \bbZ$.
\end{proof}

Next, we come to a generalization of \cite[Proposition 2.5]{MR2403307}
to coherent ringed schemes.
Let~$\sfT$ be a triangulated category
with a t-structure~$(\sfT^{\le 0},\sfT^{\ge 0})$,
and~$\sfC \subset \sfD(X,\order{A})$ be a triangulated subcategory.
A functor~$\Phi \colon \sfC\to \sfT$ has \emph{finite cohomological amplitude}
if there are~$a,b$ such that 
\begin{equation}
    \Phi(\sfC \cap \sfD^{\ge 0}(X,\order{A})) \subset \sfT^{\ge a}, \quad \text{and} \quad
    \Phi(\sfC \cap \sfD^{\le 0}(X,\order{A})) \subset \sfT^{\le b}.
\end{equation}
We refer to Section 2.3 of \emph{op.~cit.} for a short introduction to (bounded) t-structures.
\begin{lemma}\label{lem:finite-cohomological-amplitude}
    Let~$(X,\order{A})$ be a coherent ringed scheme, where~$X$ is a quasi-projective variety.
    Let~$\sfT$ be a triangulated category which admits a bounded t-structure.
    Then every functor~$\Phi\colon \sfD^\perf(X,\order{A})\to \sfT$ has finite cohomological amplitude.
\end{lemma}

\begin{proof}
 The proof of \cite[Proposition 2.5]{MR2403307} needs only a small modification.
 Let~$\calO_{X}(1)$ be an ample line bundle on~$X$,
 and denote by~$\order{A}(n) = \order{A}\otimes_X \calO_X(n)$.
 Since~$\sfT$ has a bounded t-structure, and~$X$ is quasi-projective,
 one finds~$a,b \in \bbZ$ such that~$\Phi(\order{A}(n))\in \sfT^{[a,b]}$ for all~$n\in \bbZ$.
 The remainder works as in \cite[Proposition 2.5]{MR2403307}.
\end{proof}

\subsection*{Induced semiorthogonal decompositions}
Now, we may look at the question when
semiorthogonal decompositions on~$\sfD(X,\order{A})$
and its bounded versions are induced from each other.
The exposition is close to \cite[Chapter 4]{MR2801403}.
In order to compare semiorthogonal decompositions we need the notion fo compatibility from \cite[Section 3]{MR2801403} for exact functors.
\begin{definition}
    Let~$\Phi\colon \sfT\to \sfT^\prime$ be an exact functor
    between triangulated categories with
    semiorthogonal decompositions~$\sfT = \langle \sfA_1,\ldots, \sfA_m\rangle$
    and~$\sfT^\prime = \langle \sfA^\prime_1,\ldots, \sfA^\prime_m\rangle$.
    We say that~$\Phi$\emph{ is compatible with the semiorthogonal decompositions} if~$\Phi(\sfA_k)\subset \sfA^\prime_k$.
\end{definition}
Throughout the section
we let~$\mff\colon (X,\order{A})\to S$ be a morphism
of coherent ringed schemes.
The first lemma follows as in \cite[Proposition 4.1]{MR2801403}
using \Cref{lem:perfect-is-homologically-finite}.
\begin{lemma}\label{lem:induced-sod-bounded-to-perfect}
    Let~$\bd(X,\order{A}) = \langle \sfA_1,\ldots, \sfA_m\rangle$ be a strong~$S$-linear semiorthogonal decomposition.
    If~$\order{A}$ has finite global dimension or~$X = \Spec\field$,
    then there is a unique~$S$-linear decomposition
    \begin{equation}
        \sfD^\perf(X,\order{A}) = \langle \sfA_1^\perf,\ldots, \sfA_m^\perf\rangle,
    \end{equation} 
    which is compatible with the inclusion~$\sfD^\perf(X,\order{A})\subset \bd(X,\order{A})$.
    The components are given by~$\sfA_i^\perf = \sfA_i \cap \sfD^\perf(X,\order{A})$.
\end{lemma}

The next step was already provided for~$\infty$-enhanced categories by \cite[Lemma 3.12]{MR3948688}.
We restrict ourselves to the derived category of coherent ringed schemes.
Recall from \eqref{eq:split-triangles-from-sod} that a semiorthogonal decomposition comes with
distinguished triangles.
This gives rise to \emph{projection functors}~$\pr_\ell\colon \sfT\to \sfT$,~$T\mapsto A_\ell$. 
If the decomposition is~$S$-linear, it follows from \cite[Lemma 3.1]{MR2801403}
that~$\pr_\ell(M\otimes_X \bfL f^*\calF) = \pr_\ell(M)\otimes_X\bfL f^*\calF$
for~$M\in \sfT$,~$\calF\in \sfD^\perf(S)$. 
In other words, the projection functors are~$S$-linear.
\begin{lemma}\label{lem:induced-sod-perfect-to-unbounded}
    Assume that~$\sfD^\perf(X,\order{A}) = \langle \sfA_1,\ldots, \sfA_m\rangle$ is an~$S$-linear semiorthogonal decomposition. 
    \begin{enumerate}[label = \roman*)]
        \item There is a unique~$S$-linear semiorthogonal decomposition~$\sfD(S,\order{A}) = \langle \widehat{\sfA}_1,\ldots, \widehat{\sfA}_m\rangle$
        compatible with the inclusion~$\sfD^\perf(X,\order{A})\subset \sfD(S,\order{A})$.
        \item If the semiorthogonal decomposition was induced by one of~$\bd(X,\order{A})$,
        where the projection functors~$\pr_i \colon \bd(X,\order{A})\to \bd(X,\order{A})$ have finite cohomological amplitude,
        then the semiorthogonal decomposition of~$\sfD(X,\order{A})$ is
        compatible with the inclusion~$\bd(X,\order{A})\subset \sfD(X,\order{A})$ as well.
    \end{enumerate}
\end{lemma}
\begin{proof}
    Part (i) carries over analogously from \cite[Proposition 4.2]{MR2801403}
    using \Cref{rem:perfect-complexes-and-compact-objects} that~$\sfD(X,\order{A})^{\compact} = \sfD^\perf(X,\order{A})$
    and~$\sfD(X,\order{A})$ is compactly generated.
    Note that~$\widehat{\sfA}_i$ is obtained from~$\sfA_i^\perf \subset \sfD(S,\order{A})$ 
    as the smallest triangulated category closed
    under (arbitrary) direct sums and cones
    containing~$\sfA_i^\perf$.
    Part (ii) also follows as in \cite[Proposition 4.2]{MR2801403},
    because every object in~$\bd(X,\order{A})$ admits a resolution
    by locally projective~$\order{A}$-modules, see \cite[Proposition 3.7]{MR3695056}.
\end{proof}

The projection functors~$\text{pr}_i\colon \sfA^\perf_i\to \sfA_i^\perf$ have finite cohomological amplitude by \Cref{lem:finite-cohomological-amplitude}.
Therefore,
as in \cite[Proposition 4.3]{MR2801403},
we obtain an induced semiorthogonal decomposition
on~$\sfD^-(X,\order{A})$.
\begin{lemma}\label{lem:induced-sod-unbounded-to-bounded-above}
Assume that~$\sfD^\perf(X,\order{A}) = \langle\sfA^\perf_1,\ldots, \sfA^\perf_m\rangle$ is an~$S$-linear semiorthogonal decomposition.
 There is a unique~$S$-linear semiorthogonal decomposition
\begin{equation}
    \sfD^{-}(S,\order{A}) = \langle \sfA_1^{-},\ldots, \sfA_m^{-}\rangle,
\end{equation}
with components~$\sfA_i^{-} = \widehat{\sfA}_i \cap \sfD^-(S,\order{A})$,
compatible with the embeddings~$\sfD^\perf(S,\order{A})\subset \sfD^-(S,\order{A})\subset \sfD(S,\order{A})$.
\end{lemma}

\subsection*{Base change of semiorthogonal decompositions}
We generalize the base change formulas \cite[\S 5]{MR2801403}
for semiorthogonal decompositions to the noncommutative setting.
For~$\sfD^\perf(X,\order{A})$ and~$\sfD(X,\order{A})$ this can be seen as a special case
of \cite[Lemma 3.15]{MR3948688}.
There have been several generalizations of Kuznetsov's base change formula,
notably \cite[Theorem 3.17]{MR4292740}, and \cite[Theorem 3.5]{2002.03303} weakening the assumptions
on the schemes.
Besides the base change formula \cite[Theorem 2.46]{MR2238172}
for certain Azumaya varieties,
we are not aware of results for the bounded derived category~$\bd(X,\order{A})$.
However, most of the proofs can be adapted from the commutative case.

We specialize to the situation mentioned in \Cref{rem:base-change-structure-morphism}
Assume that~$h \colon T\to S$ is a morphism of schemes
and~$\mff\colon (X,\order{A})\to S$ is a morphism of coherent ringed schemes.
Consider the base change  
\begin{equation}\label{eq:reference-nc-base-change}
\begin{tikzcd}
    (X_T,\order{A}_T)\arrow[d, "\mff_T"]\arrow[r, "\mfh_T"]& (X,\order{A})\arrow[d, "\mff"]\\
    T\arrow[r,"h"] & S
\end{tikzcd}
\end{equation}
of~$\mff$ along~$h\colon T \to S$.
\Cref{lem:noncommutative-fibre-product} implies that~$\order{A}_T = h_T^* \order{A}$.
Then, there is a natural transformation of functors
\begin{equation}
       \bfL h^*\circ\bfR \mff_{*} \Rightarrow \bfR\mff_{T,*}\circ\bfL\mfh_T^*.
   \end{equation}
from~$\sfD(X,\order{A})$ to~$\sfD(T)$.

This follows from the push-pull adjunction \cite[Lemma A.6]{MR4554471}.
Then, the construction of the natural transformation translates verbatim from \cite[{\href{https://stacks.math.columbia.edu/tag/02N6}{Section 02N6} and \href{https://stacks.math.columbia.edu/tag/08HY}{Remark 08HY}}]{stacks-project}.

\begin{definition}
    A morphism~$h\colon T \to S$ is \emph{faithful} for~$\mff$
if~$\bfL h^*\circ \bfR\mff_* \Rightarrow \bfR \mff_{T,*}\circ\bfL\mfh_T^*$ from~$\sfD(X,\order{A})$ to~$\sfD(T)$ is an equivalence.
\end{definition}
If~$\order{A}$ is flat over~$X$, a K-flat resolution~$F^\bullet$ of an~$\order{A}$-module~$M\in \sfD(X,\order{A})$ is also flat over~$X$.
Hence, on has~$\bfL \mfh_T^{*} = \bfL h_T^{*}$, and we can appeal to commutative base change, to obtain the following.
\begin{lemma}\label{lem:faithful-base-change}
    If~$\order{A}$ is flat over~$X$
    and~$f\colon X \to S$ is flat,
    then each morphism~$h\colon T \to S$ is faithful for~$\mff$.    
\end{lemma}
In the following we will always assume that~$h\colon T\to S$
is faithful for~$\mff\colon (X,\order{A})\to S$.

\paragraph{Perfect complexes.}
For base changing perfect complexes, one uses that 
$\sfD^\perf(X_T,\order{A}_T)$ is generated by `box tensors' of the components
$\sfD^\perf(T)$ and~$\sfD^\perf(X,\order{A})$
inside the unbounded derived category~$\sfD(X_T,\order{A}_T)$.
This holds in the full generality of~$\infty$-categories by \cite[Lemma 2.7]{MR3948688}.
We provide the small modifications to apply the proof of \cite[Lemma 5.2]{MR2801403}.

\begin{lemma}\label{lem:base-change-perfect-category}
    The category~$\sfD^\perf(X_T,\order{A}_T)$ is
    the minimal triangulated subcategory of~$\sfD(X,\order{A}_T)$
    closed under taking direct summands 
    which is generated by the objects
    \begin{equation}
        \bfL \mfh_T^* M\otimes_{T}^\bfL \bfL f_T^*\calF,\quad
        \text{where }M\in \sfD^\perf(S,\order{A})\text{ and } \calF\in \sfD^\perf(T).
    \end{equation}
\end{lemma}
\begin{proof}
    Since~$h\colon T\to S$ is a quasi-projective morphism,
    every coherent~$\calO_{X_T}$-module~$\calF$ admits a surjection
   ~$f_T^*\calO_h(n)^{\oplus k}\twoheadrightarrow \calF$ for some~$n,k \in \bbZ$,
    where~$\calO_h(1)$ is an~$h$-ample line bundle on~$T$. 
    Note that by \cite[{\href{https://stacks.math.columbia.edu/tag/0893}{Lemma 0893}}]{stacks-project},
    the line bundle~$f_T^*\calO_h(1)$ is~$h_T$-ample.

    Therefore, every coherent~$\order{A}_T$-module~$M$
    admits a surjection~$f_T^*\calO_h(n)^{\oplus k}\otimes_T \order{A}_T \twoheadrightarrow M$ of~$\order{A}_T$-modules,
    where we use the right~$\order{A}_T$-module structure on~$M$.
    Since~$\mfh_T$ is strict,~$\order{A}_T$ is the pullback 
    of (the locally projective~$\order{A}$-module)~$\order{A}$.
    For this reason we find for every~$M\in \sfD^\perf(X_T,\order{A}_T)$ a bounded above
    locally projective~$\order{A}_T$-resolution~$P^\bullet \to M$ such that 
    for each~$P^i \cong \mfh_T^* Q^{i} \otimes_T f_T^*\calE^{i}$,
    where~$Q^{i}$ is locally projective, 
    and~$\calE^{i}$ is a locally free~$\calO_T$-module.
    This allows us to proceed as in \cite[Lemma 5.2]{MR2801403}.
\end{proof}

With this lemma, we are ready to define the 
base change of admissible subcategories in the category of perfect complexes.
Given an~$S$-linear semiorthogonal decomposition
\begin{equation}\label{eq:sod-perfect-S-linear}
    \sfD^\perf(X,\order{A}) = \langle\sfA_1^\perf,\ldots, \sfA_m^\perf \rangle,
\end{equation}
we define~$\sfA_{iT}^\perf$ to be the smallest triangulated subcategory,
closed under direct summands containing all objects of the
form~$\bfL \mfh_T^* M\otimes_{X_T}^\bfL \bfL f_T^*\calF$for $M\in \sfA_i^\perf$ and $\calF\in \sfD^\perf(T)$.

\begin{proposition}\label{prop:base-change-perfect}
From the~$S$-linear semiorthogonal decomposition \eqref{eq:sod-perfect-S-linear}
one obtains a~$T$-linear semiorthogonal decomposition
\begin{equation}
    \sfD^\perf(X_T,\order{A}_T) = \langle \sfA_{1T}^\perf,\ldots, \sfA_{mT}^\perf\rangle
\end{equation}
compatible with~$\bfL\mfh_T^*$.
\end{proposition}

\begin{proof}
    Using \Cref{lem:rel-criterion-semiorthogonality},
    it follows from a standard manipulations
    that~$\sfA_{jT}^\perf \subset {\sfA_{iT}^\perf}^\perp$.
    The construction of~$\sfA_{iT}^\perf$ implies~$T$-linearity and~$\bfL\mfh_T^*(\sfA_i^\perf)\subseteq \sfA_{iT}^\perf$. 
    Generation follows \Cref{lem:base-change-perfect-category}.
\end{proof}

\paragraph{The unbounded derived category.}
We extend \cite[Proposition 5.3]{MR2801403} to our setting.
\begin{proposition}\label{prop:unbounded-base-change}
    Let~$\sfD(X,\order{A}) = \langle \widehat{\sfA}_1,\ldots, \widehat{\sfA}_m\rangle$ be an~$S$-linear semiorthogonal decomposition.
    Then
    \begin{equation}\label{eq:unbounded-base-change}
        \sfD(X_T,\order{A}_T) = \langle \widehat{\sfA}_{1,T},\ldots, \widehat{\sfA}_{m,T}\rangle
    \end{equation}
    is a~$T$-linear semiorthogonal decomposition compatible with~$\bfL \mfh_{T,*}$ and~$\bfL \mfh_T^*$.
    Moreover, if the decomposition is induced from a semiorthogonal decomposition \eqref{eq:sod-perfect-S-linear} on~$\sfD^\perf(X,\order{A})$,
    the semiorthogonal decomposition \eqref{eq:unbounded-base-change} is compatible with the one constructed in \Cref{prop:base-change-perfect}.
\end{proposition}
\begin{proof}
    Recall from the proof of \Cref{lem:induced-sod-perfect-to-unbounded}
    that by construction~$\widehat{A}_{iT}$ is obtained
    as the closure of~$\sfA_{iT}^\perf$ in~$\sfD(X_T,\order{A}_T)$
    under direct sums and iterated cones.
    In particular, it contains all homotopy colimits
    of perfect complexes in~$\sfA_{iT}^\perf$.
    The compatibility with the pullback follows now directly from~$\bfL \mfh_T^* \sfA_{i}^\perf\subset \sfA_{iT}^\perf \subset \widehat{\sfA}_{iT}$,
    where the first inclusion follows from the compatible base change in \Cref{prop:base-change-perfect}.
    
    Next consider the pushforward. By \Cref{lem:base-change-perfect-category}, we can consider~$M \in \sfA_{i}^\perf$ and~$\calF \in \sfD^\perf(T)$.
    It is to show that~$\bfR \mfh_{T,*} (\bfL \mfh_T^* M \otimes_{X_T}^\bfL \bfL f_T^* \calF) \in \widehat{\sfA}_i$. 
    Because~$\bfL \mfh_T^* M \otimes_{X_T}^\bfL \bfL f_T^* \calF\cong \bfL\mfh_T^*M\otimes_{\order{A}_T}^\bfL \bfL \mff_T^*\calF$
    in $\bd(X_T,\order{A}_T)$, we can use the noncommutative projection formula \cite[Proposition A.6]{MR4554471} to obtain
    \begin{equation}
        \bfR \mfh_{T,*} (\bfL\mfh_T^*M\otimes_{\order{A}_T}^\bfL \bfL \mff_T^*\calF)
        =  M \otimes_{\order{A}}^\bfL (\bfR \mfh_{T,*}\bfL f_T^* \calF).
    \end{equation}
    The right hand side lies in $\widehat{A}_i$
    as~$\bfR h_{T,*}\bfL f_T^* \calF \cong \bfL f^*\bfR h_* \calF \in \sfD(X)$,
    because~$h$ is faithful for~$\mff$,
    and~$\widehat{\sfA}_i$ is~$S$-linear.
    Note that by \Cref{rem:extended-S-linearity},
    each component is~$S$-linear for pullbacks 
    of all quasi-coherent sheaves~$\calF \in \sfD(S)$.
\end{proof}
Similarly to \cite[Proposition 5.3]{MR2801403},
we can give a more precise description of
the components~$\widehat{\sfA}_{iT}$.
For~$M \in \sfD(X_T,\order{A}_T)$ we find
that~$M \in \widehat{\sfA}_{iT}$
if and only if $\bfR\mfh_{T,*} (M\otimes_{T}^{\bfL} \bfL f_T^*\calF) \in \widehat{\sfA}_i$
for all~$\calF \in \sfD^\perf(T)$.
With statement (ii) of the next lemma this can be shown as in \cite[Proposition 5.3]{MR2801403}.

\begin{lemma}\label{lem:eventually-bounded-zero-sufficient}
    Let~$\mff \colon (X,\order{A})\to (Y,\order{B})$ be a 
    morphism of coherent ringed schemes, 
    such that~$f$ is quasi-projective~$M \in \sfD(X,\order{A})$
    and~$\calO_{f}(1)$ an~$f$-ample line bundle on~$X$.
    \begin{enumerate}[label = \roman*)]
        \item Then~$M \in \sfD^{[p,q]}(X,\order{A})$ if and only if there is a sequence~$\calO_{f}(k_1)\to \calO_{f}(k_2)\to\ldots~$ with~$k_i\to \infty$
        such that~$\hocolim \bfR \mff_*(M \otimes_X \calO_{f}(k_i))\in \sfD^{[p,q]}(Y,\order{B})$.
        \item Moreover,~$M = 0$ if and only if there is a sequence~$\calO_{f}(k_1)\to \calO_{f}(k_2)\to\ldots~$ with~$k_i\to \infty$
        such that~$\hocolim \bfR \mff_*(M \otimes_X \calO_{f}(k_i)) = 0$.
    \end{enumerate}
\end{lemma}
\begin{proof}
    Given a quasi-coherent~$\order{A}$-module~$M$,
    we have~$M \in \sfD^{[p,q]}(X,\order{A})$ 
    if and only if~$M \in \sfD^{[p,q]}(X)$,
    because the cohomology sheaf of~$M$
    does not depend on its~$\order{A}$-module structure.
    By \cite[Lemma 5.4]{MR2801403}, this holds
    if and only if~$\hocolim \bfR f_*(M \otimes_X \calO_{f}(k_i))\in \sfD^{[p,q]}(Y)$.
    Since~$\bfR \mff_* = \bfR f_*$, we find in particular
    that~$\hocolim \bfR \mff_*(M \otimes_X \calO_{f}(k_i))\in \sfD^{[p,q]}(Y,\order{B})$.
    Similarly, Lemma 5.4 of \emph{op.~cit.} implies (ii).
\end{proof}

\begin{remark}
    With the assumption that the projection functors in the semiorthogonal decompositions of \Cref{prop:base-change-perfect} have finite cohomological amplitude
    one obtains a semiorthogonal decomposition of the bounded above derived category~$\sfD^-(T,\order{A}_T) = \langle \sfA_{1T}^-,\ldots, \sfA_{mT}^-\rangle$
    compatible with derived pushforward and pullback.
    The~$T$-linear components are~$\sfA_{iT}^- = \widehat{\sfA}_{iT} \cap \sfD^{-}(T,\order{A}_T)$.  
\end{remark}

\paragraph{The bounded derived category.}
We are now ready to formulate the main statement of this section,
the generalization of \cite[Theorem 5.6]{MR2801403} to
coherent ringed schemes. 
\begin{theorem}\label{thm:main-base-change-theorem}
    Let~$\mff\colon (X,\order{A})\to S$ and~$h\colon T\to S$ be morphisms of coherent ringed schemes.
    Assume that~$\bd(X,\order{A}) = \langle\sfA_1,\ldots, \sfA_m \rangle$
    is an~$S$-linear strong semiorthogonal decomposition
    such that the projection functors have finite cohomological amplitude
    and~$h$ is faithful for~$\mff$.
    If~$\order{A}$ has finite global dimension, then
    \begin{equation}
        \bd(X_T,\order{A}_T) = \langle \sfA_{1,T},\ldots, \sfA_{m,T}\rangle
    \end{equation}
    is a~$T$-linear semiorthogonal decomposition compatible with 
    \begin{itemize}
        \item the induced semiorthogonal decomposition on~$\sfD(X_T,\order{A}_T)$ and~$\sfD^{-}(X_T,\order{A}_T)$,
        \item the pullback~$\bfL\mfh_T^*\colon \bd(S,\order{A})\to \sfD^-(X_T,\order{A}_T)$
        and the pushforward~$\bfR\mfh_{T,*}\colon \bd(X_T,\order{A}_T)\to \sfD(S,\order{A})$.
    \end{itemize}
\end{theorem}

The proof follows as in \cite[Theorem 5.6]{MR2801403}
with the following lemma for the approximation of the 
pushforward of bounded quasi-coherent complexes~$M \in \sfD^{[p,q]}(X,\order{A})$.
\begin{lemma}
    Let~$\mff \colon (X,\order{A})\to (Y,\order{B})$ be a 
    morphism of coherent ringed schemes, 
    such that~$f$ is quasi-projective~$M \in \sfD(X,\order{A})$
    and~$\calO_{f}(1)$ an~$f$-ample line bundle on~$X$.
    If~$M \in \sfD^{[p,q]}(X,\order{A})$ and~$k \gg 0$,
    then there is a direct system~$\{N_m\}$ in~$\sfD^{[p,q]}(Y,\order{B})$
    such that~$\bfR \mff_* (M\otimes_{X}\calO_f(k)) \cong \hocolim N_m$.
\end{lemma}
Since every quasi-coherent~$\order{A}$-module is a colimit of its coherent submodules,
one can use the same argument as in \cite[Lemma 2.20]{MR2801403} to prove this lemma.

\renewcommand*{\bibfont}{\small}
\printbibliography

\emph{Thilo Baumann}, \url{thilo.baumann@uni.lu} \\
Department of Mathematics, Universit\'e de Luxembourg, 6, avenue de la Fonte, L-4364 Esch-sur-Alzette, Luxembourg
\end{document}